\documentclass{article}
\usepackage{graphicx} % Required for inserting images
\usepackage{setspace}
\usepackage{amsthm,amsmath,amssymb}
\usepackage{graphicx}
\usepackage{mathrsfs}
\usepackage{tikz}
\usetikzlibrary{cd}
\usepackage{tikz-cd}
\usepackage{geometry}
\usepackage{amsfonts}
\usepackage[utf8]{inputenc}
\usepackage{amssymb}
\usepackage{amsmath}
\usepackage{amsfonts}
\numberwithin{equation}{section}
\usepackage{subfigure}
\usepackage{graphicx}
\usepackage{amsmath,amscd}
\everymath{\displaystyle}
\usepackage{indentfirst} 
\usepackage{enumerate}
\usepackage[colorlinks]{hyperref}
\usepackage{bm}
\usepackage{algorithm}
\usepackage{algorithmic}
\usepackage{mathrsfs}
\usepackage{esint}
\usepackage[toc,page,title,titletoc,header]{appendix}
\newtheorem{theorem}{Theorem}[section]
\newtheorem{prop}[theorem]{Proposition}
\newtheorem{remark}[theorem]{Remark}
\newtheorem{definition}[theorem]{Definition}
\newtheorem{lemma}[theorem]{Lemma}
\newtheorem{example}[theorem]{Example}
\newtheorem{cor}[theorem]{Corollary}

\everymath{\displaystyle}
\usepackage{graphicx} % Required for inserting images
\title{A new interpretation of Jimbo's formula for Painlev\'e VI}

\author{Zikang Wang, Yuancheng Xie and Xiaomeng Xu}

\date{}

\newcommand{\Addresses}{{
  \bigskip
  \footnotesize
\noindent \textsc{School of Mathematical Sciences, Peking University, Beijing 100871, China}\par\nopagebreak
  \textit{E-mail address}: \texttt{wzikang@stu.pku.edu.cn}
}\\
\\
\footnotesize
\noindent \textsc{
School of Mathematical Sciences \& Beijing International Center
for Mathematical Research, Peking University, Beijing 100871, China}\par\nopagebreak
  \textit{E-mail address}: \texttt{xieyuancheng@bicmr.pku.edu.cn}\\
\\
\footnotesize
\noindent \textsc{
School of Mathematical Sciences \& Beijing International Center
for Mathematical Research, Peking University, Beijing 100871, China}\par\nopagebreak
  \textit{E-mail address}: \texttt{xxu@bicmr.pku.edu.cn}
}

\begin{document}

\maketitle

\begin{abstract}
In this paper, we first give a new interpretation of Jimbo's boundary condition for the generic Painlev\'e VI transcendents, as the shrinking phenomenon in long time behaviour of the Jimbo-Miwa-Mori-Sato equation with rank $n=3$. We then interpret Jimbo's asymptotic and monodromy formula from the viewpoint of the isomonodromy deformation with respect to irregular singularities.
\end{abstract}

\section{Introduction}
%Let us consider the differential equation for a $n\times n$matrix valued function $\Phi(u)$ with $n$ complex variables $u_1,...,u_n$,
%\begin{eqnarray} \label{isoeq} \frac{\partial}{\partial u_k} \Phi(u)=[{\rm ad}^{-1}_{u}{\rm ad}_{E_{k}}\Phi(u),\Phi(u)], \quad  \text{for every $k=1,...,n$}. \end{eqnarray} 
%Here $E_k$ is the $n\times n$ diagonal matrix  whose $(k,k)$-entry is $1$ and other entries are $0$. Note that ${\rm ad}_{E_{k}}\Phi(u)$  takes values in the space  ${\frak {gl}}_n^{{\rm od}}(\mathbb{C})$  of off-diagonal matrices and that ${\rm ad}_u$ is invertible when restricted to ${\frak {gl}}_n^{{\rm od}}(\mathbb{C})$, where we have abused notation by taking $u={\rm diag}(u_1,\ldots,u_n)$.

In recent years there have been considerable interests in analyzing the differential equations 
for an 
$n\times n$
matrix-valued function 
$\Phi(u)=(\phi_{ij}(u))_{n\times n}$ depending on $n$ complex variables $u_1,...,u_n$,
\begin{equation}\label{isoeq}
\begin{aligned}
\frac{\partial}{\partial u_k} \phi_{k k}(u)
& =  0, \ \ \text{ for } k=1,...,n,\\
\frac{\partial}{\partial u_k} \phi_{i j}(u)
& =  
\left(
\frac{1}{u_k - u_j}-
\frac{1}{u_k - u_i} \right) 
\phi_{i k}(u)\phi_{k j}(u),
\quad
\ \ \text{ for } i, j \neq k,\\
\frac{\partial}{\partial u_k} \phi_{i k}(u) 
& = 
\sum_{j \neq k} 
\frac{\delta_{ij}\phi_{kk}-\phi_{i j}(u)}{u_k - u_j}
\phi_{j k}(u),
\quad \ \ \text{ for } 1\le i \neq k\le n,\\
\frac{\partial}{\partial u_k} \phi_{k j}(u)
& = 
\sum_{i \neq k} 
\phi_{k i}(u)
\frac{\phi_{i j}(u)-
\delta_{i j}\phi_{k k}}{u_k - u_i},
\quad \ \ \text{ for } 1\le j \neq k\le n.
\end{aligned}
\end{equation}

The system \eqref{isoeq} 
first appeared in the paper of 
Jimbo-Miwa-Mori-Sato \cite{JMMS1980}, 
and are also special cases 
of the equations of Jimbo-Miwa-Ueno \cite{JMU1981I}, as the isomonodromic deformation equation of a linear system with Poincar\'{e} rank $1$. 
Following Miwa \cite{Miwa1981}, 
its solutions $\Phi(u)$ 
have the strong Painlev\'{e} property: 
they are multi-valued meromorphic functions 
of $u_1,...,u_n$ and 
the branching occurs 
when $u$ moves along a loop around 
the fat diagonal
\[\Delta=
\{(u_1,...,u_n)\in \mathbb{C}^n
~|~
u_i = u_j, \text{for some $i\neq j$} \}.\]
Thus, according to the original idea of Painlev\'{e}, they may define a new class of special functions. Later on, it was shown by Harnad \cite{Harnad1994} that there is a duality of the JMMS equations that relates \eqref{isoeq} to the Schlesinger equations, and thus for $n=3$ \eqref{isoeq} is equivalent to the Painlev\'e VI equation (see Section \ref{sec:Harnadduality} for more details). Therefore, it can be seen as a higher rank Painlev\'e equation.

Since then, the sources of the interests in \eqref{isoeq} become
quite diverse, which include many subjects in mathematics and physics: the particular case (with skew-symmetric $\Phi(u)$) was studied by Dubrovin \cite{Dubrovin1996} in relation to the Gromov-Witten theory, and in general (with $\Phi(u)\in \frak{gl}_n(\mathbb{C})$) by Boalch \cite{Boalch2001, Boalch2006} in relation to complex reflections and Poisson-Lie groups; the system \eqref{isoeq} appeared in the work \cite{Bridgeland-ToledanoLaredo2012} of Bridgeland and Toledano Laredo in relation to the stability conditions; it is a time-dependent Hamiltonian system, whose quantization, following the work \cite{Reshetikhin1992} of 
Reshetikhin, is related to the Knizhnik--Zamolodchikov equation in the conformal field theory. 

Despite of the many applications, the behaviour of its solutions $\Phi(u)$ at critical points and monodromy problems are only studied recently in \cite{xu2019closure1, TangXu}. 
In this paper, we prove that in the case $n=3$, the boundary value and monodromy formula of equation \eqref{isoeq}, i.e., the following Theorem \ref{Thm:termwise} and \ref{thm: introcatformula}, recover Jimbo's formula for the asymptotics of Painlev\'e VI transcendents and the monodromy of the associated linear system respectively. 

The main result of \cite{TangXu} is the following shrinking phenomenon of solutions of
the nonlinear equation \eqref{isoeq} in long time $u_1,...,u_n$ behaviour. For example, for any generic solution $\Phi(u)$ with
any two eigenvalues 
$\lambda^{(n-1)}_1(u)$ and 
$\lambda^{(n-1)}_2(u)$ of 
the upper left $(n-1)\times (n-1)$ submatrix of $\Phi(u)$,
we have
\begin{align*}
\underset{u_n\rightarrow\infty}
{{\lim}}
|{\rm Re}(
\lambda^{(n-1)}_1(u)-
\lambda^{(n-1)}_2(u))|<1.
\end{align*}
To be more precise, given any $n\times n$ matrix $A$, 
we denote by $\delta_k A$ the
the upper left $k\times k$ submatrix and 
the diagonal part of the $n\times n$ matrix $A$, i.e.
\begin{eqnarray}\label{delta}
  (\delta_k A)_{i j} 
  & := & 
  \left\{\begin{array}{ll}
    A_{i j}; \qquad & 1 \leqslant i, j \leqslant k \quad
    \text{or}\quad i = j,\\
    0 ; \qquad &  \text{otherwise.}
  \end{array}\right.
\end{eqnarray} 
The following theorem was first proved for skew-Hermitian valued solutions in \cite{xu2019closure1}, and later on generalized to generic $\frak{gl}_n$ valued solutions in \cite{TangXu}. Note that for skew-Hermitian valued solution $\Phi(u)$, the boundary condition \eqref{boundcondtion} is empty.
\begin{theorem}
\label{Thm:termwise}\cite{TangXu}
For almost every solution $\Phi(u)=\Phi_n(u)$ of 
the isomonodromy equation \eqref{isoeq}, there exist $n \times n$ matrix-valued functions
$\Phi_{k}(u_1,...,u_k)$ for $k=1,...,n-1$ such that $\Phi_0:=\Phi_1$ is constant and
for $2\leqslant k \leqslant n$ we have
\begin{align} 
\label{limit1}
&\underset{u_k \rightarrow \infty}{\lim} 
\delta_{k - 1} \Phi_k  = 
\delta_{k - 1} \Phi_{k - 1},\\ 
\label{limit2}
&\underset{u_k \rightarrow \infty}{\lim} 
\left(
\frac{u_k-u_{k-1}}{u_{k-1}-u_{k-2}}
\right)^{
\delta_{k - 1} \Phi_{k - 1}} \Phi_k \left(
\frac{u_k-u_{k-1}}{u_{k-1}-u_{k-2}}
\right)^{
-\delta_{k - 1} \Phi_{k - 1}} =  \Phi_{k - 1}, \ k=3,...,n \\
&
\left(
{u_2-u_1}
\right)^{
\delta_{1} \Phi_{1}} \Phi_2 \left(
{u_2-u_{1}}
\right)^{
-\delta_{1} \Phi_{1}} =  \Phi_{1}
\end{align}
and
\begin{equation}
\label{boundcondtion}
|{\rm Re} (
\lambda^{(k-1)}_{i} - 
\lambda^{(k-1)}_{j}
) | < 1,
\quad
\text{for every $1\leqslant i,j\leqslant k-1$},
\end{equation}
where 
$\{\lambda^{(k-1)}_{i}\}_{i=1,\ldots,k-1}$
are the eigenvalues of the
upper left $(k-1)\times (k-1)$ submatrix
of $\Phi_{0}$.

Conversely, for any constant matrix $\Phi_0$ 
satisfying the boundary condition \eqref{boundcondtion},
there exists a unique solution $\Phi_n(u;\Phi_0)$ and a series of functions
with $\Phi_{n-1},\ldots,\Phi_{1}$
such that \eqref{limit1} and \eqref{limit2} hold.
\end{theorem}
The regularized limit $\Phi_0$ is called the boundary value of $\Phi(u)$ (at the limit $\frac{u_{k+1}-u_{k}}{u_{k}-u_{k-1}}\rightarrow \infty$), and \eqref{boundcondtion} is called the boundary condition. Following Theorem \ref{Thm:termwise}, the boundary value $\Phi_0$ gives a parameterization of the generic solutions of \eqref{isoeq}. Most importantly, the monodromy data of the associated linear system has a closed formula in terms of $\Phi_0$ recalled as follows.

Let us consider the $n\times n$ linear system of meromorphic differential equation for a function 
$F(z,u)\in {\rm GL}_n(\mathbb{C})$ 
\begin{equation}
\label{introisoStokeseq1}
\frac{\partial}{\partial z} F(z,u)=
\Big(U + \frac{\Phi(u;\Phi_0)}{z}\Big)
\cdot F(z,u),
%\quad 
%\text{for all $k=1,...,n$},
\end{equation}
where $\Phi(u;\Phi_0)$ is the 
solution of \eqref{isoeq} with the boundary value $\Phi_0$ in the sense of Theorem \ref{Thm:termwise}. Here we have 
%abused notation by taking 
$U ={\rm diag}(u_1,\ldots,u_n)$,
when we want $U$ 
to represent a diagonal matrix according to the context. 

For any fixed 
$u\in\mathbb{C}^n
\setminus\Delta$, 
the equation \eqref{introisoStokeseq1} 
has a unique formal solution $\hat{F}(z,u)$ 
around $z = \infty$. 
Then the standard summability theory 
states that there exist certain sectorial regions around $z=\infty$, 
such that on each of these sectors there is a unique (therefore canonical) holomorphic solution with the prescribed asymptotics $\hat{F}(z,u)$. 
These solutions are in general different 
(that reflects the Stokes phenomenon), 
and the transition between them 
can be measured by a pair of Stokes matrices $S_\pm(u,\Phi(u;\Phi_0))\in{\rm GL}(n)$ 
(see Section \ref{sec:monodromyduality} for more details).
Varying $u$, the Stokes matrices 
$S_\pm(u,\Phi(u))\in {\rm GL}(n,\mathbb{C})$ of the system are locally constant (independent of $u$), and this is why the equations \eqref{isoeq} are called isomonodromy equations.

%\textcolor{blue}{Make sure the following theorem has the correct statement.}
\begin{theorem}
\label{thm: introcatformula}\cite{xu2019closure1}
For all purely imaginary parameters $u_1,...,u_n$ with ${\rm Im}(u_1)<{\rm Im}(u_2)<\cdots <{\rm Im}(u_n)$, the diagonals and sub-diagonals of the Stokes matrices $S_\pm(u,\Phi(u;\Phi_0))$ are given by
\begin{align*}
(S_+)_{k,k}  &= {\rm e}^{-{\rm i}\pi\cdot\lambda^{(k-1)}_k},\qquad (S_-)_{k,k}  = {\rm e}^{-{\rm i}\pi\cdot\lambda^{(k-1)}_k},\quad k=1,2,...,n\\
(S_+)_{k,k+1}&=2{\rm i}\pi\cdot e^{-{\rm i}\pi\cdot{\small{\lambda^{(k-1)}_{k}}}}\\
&\times \sum_{i=1}^k\frac{\prod_{l=1,l\ne i}^{k}\Gamma(1+\lambda^{(k)}_i-\lambda^{(k)}_l)}{\prod_{l=1}^{k+1}\Gamma(1+\lambda^{(k)}_i-\lambda^{(k+1)}_l)}\frac{\prod_{l=1,l\ne i}^{k}\Gamma(\lambda^{(k)}_i-\lambda^{(k)}_l)}{\prod_{l=1}^{k-1}\Gamma(1+\lambda^{(k)}_i-\lambda^{(k-1)}_l)}\cdot \Delta^{1,...,k-1,k}_{1,...,k-1,k+1}(\lambda^{(k)}_i-{\Phi_0}),\\
(S_-)_{k+1,k}& = -2\mathrm{i}\pi\cdot e^{-\mathrm{i}\pi\cdot{\small{\lambda^{(k)}_{k+1}}}}\\
&\times 
\sum_{i=1}^k \frac{\prod_{l=1,l\ne i}^{k}\Gamma(1+\lambda^{(k)}_l-\lambda^{(k)}_i)}{\prod_{l=1}^{k+1}\Gamma(1+\lambda^{(k+1)}_l-\lambda^{(k)}_i)}\frac{\prod_{l=1,l\ne i}^{k}\Gamma(\lambda^{(k)}_l-\lambda^{(k)}_i)}{\prod_{l=1}^{k-1}\Gamma(1+\lambda^{(k-1)}_l-\lambda^{(k)}_i)}\cdot  {\Delta^{1,...,k-1,k+1}_{1,...,k-1,k}(\Phi_0-{\lambda^{(k)}_i})}.
\end{align*}
% \begin{eqnarray*}
% (S_+)_{k,k+1}=2\pi\mathrm{i}\cdot e^{-\pi\mathrm{i}\cdot\frac{\small{\lambda^{(k-1)}_{k}+\lambda^{(k)}_{k+1}}}{2}} \sum_{i=1}^k\frac{\prod_{l=1,l\ne i}^{k}\Gamma(1+\lambda^{(k)}_i-\lambda^{(k)}_l)}{\prod_{l=1}^{k+1}\Gamma(1+\lambda^{(k)}_i-\lambda^{(k+1)}_l)}\frac{\prod_{l=1,l\ne i}^{k}\Gamma(\lambda^{(k)}_i-\lambda^{(k)}_l)}{\prod_{l=1}^{k-1}\Gamma(1+\lambda^{(k)}_i-\lambda^{(k-1)}_l)}\cdot \Delta^{1,...,k-1,k}_{1,...,k-1,k+1}({\Phi_0}-\lambda^{(k)}_i),\\
% (S_-)_{k+1,k}=-2\pi\mathrm{i}\cdot e^{-\pi\mathrm{i}\cdot\frac{\small{\lambda^{(k-1)}_{k}+\lambda^{(k)}_{k+1}}}{2}}
% \sum_{i=1}^k \frac{\prod_{l=1,l\ne i}^{k}\Gamma(1+\lambda^{(k)}_l-\lambda^{(k)}_i)}{\prod_{l=1}^{k+1}\Gamma(1+\lambda^{(k+1)}_l-\lambda^{(k)}_i)}\frac{\prod_{l=1,l\ne i}^{k}\Gamma(\lambda^{(k)}_l-\lambda^{(k)}_i)}{\prod_{l=1}^{k-1}\Gamma(1+\lambda^{(k-1)}_l-\lambda^{(k)}_i)}\cdot {\Delta^{1,...,k-1,k+1}_{1,...,k-1,k}({\lambda^{(k)}_i-\Phi_0})}.
% \end{eqnarray*}
where $k=1,...,n-1$, $\{\lambda^{(k)}_{i}\}_{ i=1,2,...,k}$ are the eigenvalues of left-top $k\times k$ submatrix of $\Phi_0$, $\lambda^{(k)}_{k+1} = (\Phi_0)_{k+1,k+1}$ and  $\Delta^{1,...,k-1,k}_{1,...,k-1,k+1}({\lambda^{(k)}_i-\Phi_0})$ is the $k$ by $k$ minor of the matrix $(\lambda^{(k)}_i{\rm Id}_n-\Phi_0)$ formed by the first $k$ rows and $1,...,k-1,k+1$ columns (here ${\rm Id}_n$ is the rank $n$ identity matrix). 
Furthermore, the other entries are also given by explicit expressions.
\end{theorem}

In Theorem \ref{thm: introcatformula} we used slightly different convention on $\Phi_0$ from the one used in \cite{xu2019closure1} due to the difference of the form of equation \eqref{introisoStokeseq1}.

\begin{remark}
Note that the expression of the Stokes matrices given in Theorem \ref{thm: introcatformula} is an analytic function of $\Phi_0$ for
all $\Phi_0\in\frak{gl}_n$ satisfying the condition \eqref{boundcondtion}. Its poles along the boundary, when some of the inequalities in \eqref{boundcondtion} become equalities, is closely related to the non-generic solutions of the system \eqref{isoeq}.
\end{remark}

\subsection{A new interpretation of Jimbo's formula for Painlev\'e VI}

%The six classical Painlev\'{e} equations were introduced at the turn of the twentieth century by Painlev\'{e} \cite{Painleve1902} and Gambier \cite{Gambier1910}, in a specific classification problem for second order ODEs. Since then, they have appeared in the integrable nonlinear PDEs, 2D Ising models, random matrices, topological field theory and so on. 
We refer the reader to the book of Fokas, Its, Kapaev and Novokshenov \cite{FIKN2006} for a thorough introduction to the history and developments of the study of Painlev\'{e} equations. In particular, the sixth Painlev\'{e} equation (simply denoted by PVI or Painlev\'e VI) is the nonlinear differential equation
\begin{eqnarray}\nonumber
\frac{d^2y}{dx^2}&=&\frac{1}{2}\Big[\frac{1}{y}+\frac{1}{y-1}+\frac{1}{y-x}\Big](\frac{dy}{dx})^2-\Big[\frac{1}{x}+\frac{1}{x-1}+\frac{1}{y-x}\Big]\frac{dy}{dx}\\ \label{eq:PVI}
&+&\frac{y(y-1)(y-x)}{x^2(x-1)^2}\Big[\frac{(\theta_\infty-1)^2}{2}-\frac{\theta_1^2}{2}\frac{x}{y^2}+\frac{\theta_3^2}{2}\frac{x-1}{(y-1)^2}+\frac{1-\theta_2^2}{2}\frac{x(x-1)}{(y-x)^2}\Big],
\end{eqnarray}
with complex parameters $\theta_1,\theta_2,\theta_3,\theta_\infty \in\mathbb{C}$.
% where $\theta_1,\theta_2,\theta_3,\theta_\infty$, with $
% \theta_\infty\neq 0$, {\color{red} what is this condition for?} are
% as follows
% \begin{align}\label{eq:para of PVI}
%         2\alpha = (\theta_\infty-1)^2 ,\ 2\beta = -\theta_1^2,\ 2\gamma = \theta_3^2, \ 2\delta=1-\theta_2^2.
%     \end{align}
A solution $y(x)$ of PVI has $0,1,\infty$ as critical points, and can be analytically
continued to a meromorphic function on the universal covering of $\mathbb{P}^1\setminus\{0, 1,\infty\}$. 

\subsubsection{A new interpretation of Jimbo's boundary condition}
The asymptotics for Painlev\'e VI transcendents $y(x)$ were evaluated via the isomonodromy approach for the generic case by Jimbo \cite{Jimbo1982}.
A transcendent in the generic class obtained by Jimbo has the following critical behaviour at $x = 0$
\begin{eqnarray}\label{VIasy}
  y(x)
\ \sim \
  \left\{\begin{array}{ll}
    Jx^{1-\sigma}(1+O(x^{\epsilon})), \qquad & {\rm Re}(\sigma)>0,\\
      Jx^{1-\sigma}+J_1x^{1+\sigma}+J_2x+O(x^{2-\sigma}), \qquad &  {\rm Re}(\sigma)=0, \ \sigma\neq 0,
  \end{array}\right.
\end{eqnarray} 
 where $\epsilon$ is a small positive number, 
 \begin{align*}
     J_1 = \frac{(\sigma^2-(\theta_1-\theta_2)^2)(\sigma^2-(\theta_1+\theta_2)^2)}{16\sigma^4J}, \qquad  J_2 = \frac{\theta_1^2-\theta_2^2+\sigma^2}{2\sigma^2},
 \end{align*} and $\sigma$ satisfies the boundary condition
\begin{equation}\label{VIbound}
 0\leq{\rm Re}(\sigma) < 1.    
\end{equation}
A detailed review of the asymptotic expansion can be found in \cite{Guzzetti2015review}. 
Thus, in the following, let us denote by $y(x;\sigma,J,\theta_1,\theta_2,\theta_3,\theta_\infty)$ a generic solution of the equation \eqref{eq:PVI} with the parameters $\theta_1,\theta_2,\theta_3,\theta_\infty$, and the asymptotics \eqref{VIasy}.

It was shown by Harnad \cite{Harnad1994} (see also \cite{Mazzocco2002}, \cite[Section 3]{Boalch2005} and \cite{Degano-Guzzetti2023} for a detailed way to do the Harnad duality), that Painlev\'{e} VI is equivalent to the equation \eqref{isoeq} with $n=3$ and suitable $\Phi(u_1,u_2,u_3)$. As a consequence, following Theorem \ref{thm: Expression of Omega in solution of PVI} from \cite{Degano-Guzzetti2023}, given a generic solution $y(x;\sigma,J,\theta_1,\theta_2,\theta_3,\theta_\infty)$, there corresponds to a family of equivalent solutions $\Phi(u_1,u_2,u_3)$ of \eqref{isoeq} with $x={(u_2-u_1)}/{(u_3-u_1)},$ and
\begin{align}\label{eq:restr of diag of phi}
    &\delta\Phi = -\text{diag}(\theta_1,\theta_2,\theta_3);\\ \label{eq:restr of eigen of phi}
   &\Phi\ \text{has distinct eigenvalues}\ = 0, \ \frac{\theta_\infty-\theta_1-\theta_2-\theta_3}{2},\ 
    \frac{-\theta_\infty-\theta_1-\theta_2-\theta_3}{2}.
\end{align}
Here $\delta\Phi$ is the diagonal part of $\Phi$, and some discrete choices are made as in \eqref{eq:restr of diag of phi}, \eqref{eq:restr of eigen of phi}.

Since the asymptotics \eqref{VIasy} of the $y(x)$ at $x=0$ amounts to the behaviour of the corresponding solution $\Phi(u)$ as 
$x={(u_2-u_1)}/{(u_3-u_1)}\rightarrow \infty$, one may expect a relation between the boundary value $\Phi_0$ of $\Phi(u)$ and the asymptotic parameters $\sigma,J$ of $y(x)$. It is indeed the case. In Section \ref{p1.4} we prove

\begin{theorem}\label{mainthm}
Let $y(x;\sigma, J, \theta_1,\theta_2,\theta_3,\theta_\infty)$ be a generic solution of the Painlev\'e VI equation \eqref{eq:PVI}, $\sigma, J \neq 0$, and $\Phi(u;\Phi_0)$ be a corresponding generic solution of \eqref{isoeq} with $n=3$. Then there exists a unique $K^0 = {\rm diag}(k_1^0,k_2^0,1)$ for $k_1^0,k_2^0\in \mathbb{C}\setminus \{0\}$ such that the entries of $3\times 3$ matrix $\Phi'_0= (K^0)^{-1} \Phi_0 K^0$ are
    \begin{align}\label{111}
        \begin{split}
            (\Phi'_0)_{ii} &= -\theta_i, \quad  i=1,2,3, \\
            (\Phi'_0)_{12} &= \frac{\theta_1-\theta_2-\sigma}{2},\\
            (\Phi'_0)_{21} &= \frac{-\theta_1+\theta_2-\sigma}{2},\\
            (\Phi'_0)_{13} &=\frac{1}{2}\cdot(-\theta_3-\theta_\infty+\sigma)-\frac{1}{8\sigma^2 J}\cdot(\theta_1-\theta_2-\sigma)(\theta_1+\theta_2-\sigma)(\theta_3+\theta_{\infty}+\sigma),\\
            (\Phi'_0)_{31}&=\frac{J}{2}\cdot(\theta_3-\theta_\infty+\sigma)-\frac{1}{8\sigma^2}\cdot(-\theta_1+\theta_2-\sigma)(\theta_1+\theta_2+\sigma)(\theta_3-\theta_\infty-\sigma),\\
             (\Phi'_0)_{23} &=\frac{1}{2}\cdot(-\theta_3-\theta_\infty+\sigma)-\frac{1}{8\sigma^2 J}\cdot(\theta_1-\theta_2+\sigma)(\theta_1+\theta_2-\sigma)(\theta_3+\theta_\infty+\sigma),\\
             (\Phi'_0)_{32}&=\frac{J}{2}\cdot(-\theta_3+\theta_\infty-\sigma)-\frac{1}{8\sigma^2}\cdot(\theta_1-\theta_2-\sigma)(\theta_1+\theta_2+\sigma)(\theta_3-\theta_\infty-\sigma).
        \end{split}
    \end{align}
    Furthermore, the condition \eqref{boundcondtion} for $\Phi_0$ becomes Jimbo's boundary condition \eqref{VIbound} for $\sigma$. 
\end{theorem}
\begin{remark}\label{rem:extend sigma}
Theorem \ref{mainthm} can be extended to the case $\sigma=0$. Following \cite{Guzzetti2006matching}, there exist solutions $y(x)$ of PVI with the asymptotic behavior
\begin{align*}
        y(x) \sim x\left[\frac{\theta_2^2-\theta_1^2}{4}\left({\rm log}x+\frac{2\Tilde{J}}{\theta_1^2-\theta_2^2}\right)^2+\frac{\theta_1^2}{\theta_1^2-\theta_2^2}\right]+ O(x^2{\rm log}^3x), \ \theta_1\ne\pm\theta_2.
    \end{align*}
Let $\Phi(u;\Phi_0)$ be the corresponding solutions of \eqref{isoeq} with $n=3$, satisfying conditions \eqref{eq:restr of diag of phi}, \eqref{eq:restr of eigen of phi}, then the entries of $3\times 3$ matrix $\Phi'_0= (K^0)^{-1} \Phi_0 K^0$ are
%\begin{equation}
\begin{align*}%\label{sigma0}
        \begin{split}
            & (\Phi'_0)_{ii} = -\theta_i, \quad  i=1,2,3, \qquad
            (\Phi'_0)_{12} = \frac{\theta_1-\theta_2}{2}, \qquad
            (\Phi'_0)_{21} = \frac{-\theta_1+\theta_2}{2},\\
            & (\Phi'_0)_{13} = 1 + \frac{(\theta_3+\theta_\infty)(\Tilde{J}-\theta_1)}{\theta_1^2-\theta_2^2}, \qquad 
             (\Phi'_0)_{31} =\frac{\theta_3-\theta_\infty}{4}(\Tilde{J}+\theta_1)-\frac{\theta_1^2-\theta_2^2}{4}, \\
             & (\Phi'_0)_{23} =1+\frac{(\theta_3+\theta_\infty)(\Tilde{J}+\theta_2)}{\theta_1^2-\theta_2^2},\qquad
             (\Phi'_0)_{32} =\frac{\theta_3-\theta_\infty}{4}(\theta_2-\Tilde{J})+\frac{\theta_1^2-\theta_2^2}{4}. 
        \end{split}
\end{align*}
%\end{equation}
\end{remark}

Note that formulae in \eqref{111}  are invertible. In particular, the parameters of $y(x)$ can be expressed by the boundary value $\Phi_0$ of $\Phi(u)$ as follows.
\begin{cor}\label{cor: invert of Phi to y}
Let $y(x;\sigma, J, \theta_1,\theta_2,\theta_3,\theta_\infty)$ be a generic solution of the Painlev\'e VI equation, $\sigma, J\ne 0$, and $\Phi(u;\Phi_0)$ the corresponding solution of \eqref{isoeq} with $n=3$. Set $\varphi_{ij} = (\Phi_0)_{ij}$, then
\begin{align}
    \theta_i &= -\varphi_{ii},\ i=1,2,3, \nonumber\\
    \sigma^2 &= {(\varphi_{11}-\varphi_{22})^2+4\varphi_{12}\varphi_{21}}, \nonumber \\     
    \theta_\infty^2 &= 4(\varphi_{12}\varphi_{21}+\varphi_{23}\varphi_{32}+\varphi_{31}\varphi_{13})+\theta_1^2+\theta_2^2+\theta_3^2-2(\theta_1\theta_2+\theta_2\theta_3+\theta_1\theta_3 ), \label{eq:def of theta infty}
    \\   
    (\theta_\infty+\theta_3-\sigma)(\theta_\infty-\theta_3-\sigma)J &= (\varphi_{13}\varphi_{31}-\varphi_{32}\varphi_{23})+ \frac{(\varphi_{22}-\varphi_{11})(2\varphi_{33}-\varphi_{11}-\varphi_{22})}{4} \nonumber \\ 
    &\quad + \frac{2}{\sigma}(\varphi_{13}\varphi_{32}\varphi_{21}-\varphi_{23}\varphi_{31}\varphi_{12})+\frac{(\theta_1^2-\theta_2^2)(\theta_3^2-\theta_\infty^2)}{4\sigma^2}. \label{eq:invertible Phi0 to J}
\end{align}
% {\color{red} delet? Thus after choosing the sign of  $\theta_\infty$ in \eqref{eq:def of theta infty}, and choosing the sign of square roots defining the $\{\theta_i\}_{i=1,2,3,\infty}$ in \eqref{eq:para of PVI}, the solution of Painlev\'e VI equation $y(x;\theta_1,\theta_2,\theta_3,\theta_\infty,\sigma, J)$, with $\sigma, J\neq 0$ and $\theta_\infty\pm\theta_3-\sigma\neq 0$ are in one-to-one correspondence with boundary value $\Phi_0$  of the solution of $3\times 3$ isomonodromy equation $\Phi(u)$ satisfying conditions \eqref{eq:restr of diag of phi}, \eqref{eq:restr of eigen of phi}, considered up to conjugation with a nonsingular diagonal matrix.}
\end{cor}

From Theorem \ref{mainthm} and Corollary \ref{cor: invert of Phi to y}, we see that the critical behavior \eqref{VIasy} of a generic solution of PVI at $x = 0$ can be completely characterized by $\Phi_0$. In particular Jimbo's parameter $\sigma$ of $y(x;\sigma, J, \theta_1,\theta_2,\theta_3,\theta_\infty)$ is given by
\[\sigma=\lambda^{(2)}_1(\Phi_0)-\lambda^{(2)}_2(\Phi_0),\]
where $\lambda^{(2)}_1$ and $\lambda^{(2)}_2$ are the eigenvalues of the upper left $2\times 2$ submatrix of the $3\times 3$ matrix $\Phi_0$. Here we assume ${\rm Re}(\lambda^{(2)}_1)\ge {\rm Re}(\lambda^{(2)}_2)$. Then the boundary conditions \eqref{boundcondtion} for $n=3$ case becomes Jimbo's condition \eqref{VIbound}. In this way, we give a new interpretation of Jimbo's condition \eqref{VIbound} as the shrinking phenomenon in the long time $u_3 \to \infty$ behaviour of the generic solution $\Phi(u_1,u_2,u_3)$ of \eqref{isoeq} for $n=3$ case.

\subsubsection{A new interpretation of Jimbo's monodromy formula}
The Painlev\'e VI equation \eqref{eq:PVI} can be realized as the isomonodromy deformation equation of a $2\times 2$ Fuchsian linear system. Then Jimbo \cite{Jimbo1982} gave the explicit expression of the monodromy of the linear system associated to a generic $y(x;\sigma, J, \theta_1,\theta_2,\theta_3,\theta_\infty)$. See Section \ref{sec:regularsystem} for more details on the Fuchsian system, its monodromy data, and see Theorem \ref{thm Jimbo's formua for leading term of PVI} for the precise statement of Jimbo's formula.

Under the Harnad duality,
on the one hand, following \cite{Balser-Jurkat-Lutz1981, Boalch2005, Degano-Guzzetti2023}, the monodromy parameters $(p_{ij}, \theta_k)$ of the $2\times 2$ dual Fuchsian system associated to the transcendent $y(x;\sigma, J, \theta_1,\theta_2,\theta_3,\theta_\infty)$ can be explicitly expressed by the $3\times 3$ Stokes matrices of $S_\pm(u,\Phi(u;\Phi_0))$;  on the other hand, following Theorem \ref{mainthm}, the boundary value $\Phi_0$ can be explicitly expressed by the parameters $(\sigma, J, \theta_1,\theta_2,\theta_3,\theta_\infty)$. Therefore, we get the following diagram under the duality
\[
\begin{CD}
\text{\Big\{Stokes matrices $S_\pm(u,\Phi(u;\Phi_0))$ \Big\}} @> \text{\cite{Degano-Guzzetti2023}, see Theorem \ref{thm: relation of pij and S+-}} >> \text{\Big\{Monodromy parameters $\{p_{ij}, \theta_k\}$ \Big\}} \\
@A \text{Theorem \ref{thm: introcatformula} for $n=3$} AA    @V\text{Theorem \ref{thm Jimbo's formua for leading term of PVI}, Jimbo's formula \cite{Jimbo1982} }VV \\
\text{$\Big\{ \text{Boundary values } \Phi_0 \Big\}$} @< \text{Theorem \ref{mainthm} }<< \text{\Big\{parameters $\sigma, J, \theta_1, \theta_2, \theta_3, \theta_\infty$ \Big\} }
\end{CD}\]
In this way, we get a new interpretation of Jimbo's formula as the monodromy formula in Theorem \ref{thm: introcatformula} of the $3\times 3$ linear system \eqref{introisoStokeseq1} with an irregular singularity under the duality. Without knowing it a priori, by composing the other three arrows, one can find ( the inverse of) the formula of Jimbo. \textcolor{red}Instead of doing this, for simplicity in Section \ref{commdiag} we check by hand the composition of the four arrows is an identity. The computation reduces to some combinatorial identities of trigonometric functions. It is rather encouraging to see those complicated formulas, involved in the four theorems/arrows in the diagram, in the end match up!

Now let us explain the main ideas that we want to convey in this article. As stressed in \cite{FIKN2006,Its-Novokshenov1986}, the solutions of Painlev\'{e} equations are seen as nonlinear special functions, because they play the same role in nonlinear mathematical physics as that of classical special functions, like Airy functions, Bessel functions, etc., in linear physics. And it is the answers of the following questions that make Painlev\'{e} transcendents as efficient in applications as linear special functions:
\begin{itemize}
\item[(a)] the parametrization of Painlev\'{e} transcendents $y(x)$ by their asymptotic behaviour at critical points;

\item[(b)] the explicit expression of the monodromy of the associated linear problem via the parametrization at critical points;

\item[(c)] the construction of the connection 
formula from one critical point to another. 
\end{itemize}

Given the many known applications and the Painlev\'e property, we believe that the transcendents $\Phi(u)$, as the higher rank analog of Painlev\'e VI, have richer structures and applications remained to be explored. And just like Painlev\'e VI case, we expect that the answers to the above problems $(a)-(c)$ for $\Phi(u)$ will play crucial roles in other problems from mathematical physics. 

The result in this paper justifies that the boundary value $\Phi_0$ in Theorem \ref{Thm:termwise} is the right parametrization of generic transcendents $\Phi(u)$. It thus justifies that Theorem \ref{Thm:termwise} and \ref{thm: introcatformula} give the answers to the problems $(a)$ and $(b)$ for the isomonodromy equation \eqref{isoeq}. Guided by the result in this paper, we would like to first translate various results of Painlev\'e VI transcendents $y(x)$ to the $3\times 3$ matrix $\Phi(u)$ via the Harnad duality, and then generalize them to arbitrary rank $\Phi(u)$ via the help of Theorem \ref{Thm:termwise} and \ref{thm: introcatformula}. For example, the connection 
formula in Problem $(c)$ between two special critical points is derived in \cite{Xu2023} as a consequence of the two theorems. The rich literature on Painlev\'e VI provides us many interesting questions that can be asked for $\Phi(u)$, including the connection 
formula, the classification of the algebraic solutions, the study of the initial value space, its quantum monodromy manifolds \cite{CMR} and so on.

\vspace{3mm}
The organization of the paper is as follows. Section \ref{JMMMH} gives a brief introduction to the isomonodromy deformation equation (the JMMS equation) of some linear systems and their Harnad duality. Section \ref{sec:regularsystem} focuses on the special case, i.e., a $2\times 2$ Fuchsian linear system with four regular singularities, and the dual $3\times 3$ linear system with one regular and one irregular singularity, whose isomonodromy equation give rise to the PVI and the equation \eqref{isoeq} for $n=3$ respectively. It summarizes the results about Jimbo's formula, as well as the correspondence between their monodromy and solutions under the duality. Section \ref{lastsec} first proves Theorem \ref{mainthm} and then gives a new interpretation of Jimbo's formula.

\subsection*{Acknowledgements}
\noindent
The authors would like to thank Qian Tang for useful discussion. The authors are supported by the National Key Research and Development Program of China (No. 2021YFA1002000) and by the National Natural Science Foundation of China (No. 12171006). Y. X. is also supported by the National Natural Science Foundation of China under the Grant No. 12301304.

\section{JMMS equations and Harnad duality}\label{JMMMH}
\subsection{Linear systems with one irregular singularity}

% The symmetries and Hamiltonian structure of Painlev\'e VI \eqref{eq:PVI}, however, is more easily seen when it is viewed as deformation condition for a linear system which we now introduce.

%We consider a linear system a little bit more general than we have seen in the Introduction which contains both sides of Harnad duality so that we can use \eqref{isoeq} to study PVI. 

Let us consider the linear ODE system
\begin{equation}\label{eq:irregularsystem}
    \frac{dF}{dz} = A(z)F, \qquad A(z) = U + \sum \limits_{i = 1}^N\frac{A_i}{z - t_i},
\end{equation}
where $A_i$'s are $n \times n$ complex matrices, and $U = \text{diag}(u_1, \dots, u_n) = \sum_{\alpha = 1}^n u_{\alpha}E_{\alpha}$ is a diagonal matrix where $E_{\alpha}$ is the $n \times n$ matrix with $(\alpha, \alpha)$ entry be $1$ and $0$ at all the other places. 

System \eqref{eq:irregularsystem} has regular (Fuchsian) singularities at $z = t_i (i = 1, \dots, N)$ and an irregular singularity of Poincar\'e rank $1$ at $z = \infty$ if $U \ne 0$ or a regular singularity at $z = \infty$ if $U = 0$ and $A_{\infty} = -\sum \limits_{i = 1}^N A_i \ne 0$. Let $\theta_{i\alpha} (\alpha = 1, \dots, n)$ be the eigenvalues of $A_i$. Let us assume the following generic conditions
\begin{equation}\label{eq:genericconditions}
\theta_{i\alpha} - \theta_{i\beta} \not\in \mathbb{Z}, \qquad \text{and} \qquad u_{\alpha} - u_{\beta} \ne 0, \quad \text{for $\alpha \ne \beta$ if $U \ne 0$,}
\end{equation}
% $\theta_{i\alpha} - \theta_{i\beta} \not\int \mathbb{Z}$ is a no logarithm condition which ensures that there is no logarithm term in the coefficients $Y_j^{(i)}$ in the expression \eqref{eq:regseries}(c.f.\cite{Mahoux1999}).
to avoid the resonant cases. Then the coefficients $A_i$ are all diagonalizable
\[A_i = G_i \Theta_i G_i^{-1}, \qquad \Theta_i = \text{diag}(\theta_{i1}, \dots, \theta_{i, n}).\]
Note that $G_i$'s are not unique, and each of them can be multiplied on the right by a diagonal matrix.

By the Riemann-Hilbert-Birkhoff correspondence (up to submanifolds where the inverse monodromy problem for \eqref{eq:irregularsystem} is not solvable), the
global analytic solutions of \eqref{eq:irregularsystem} are characterized by the associated monodromy data which includes: (1) the monodromy matrix (or ``exponent of formal monodromy''\cite{JMU1981I}) around each regular singular point; (2) the relevant Stokes matrices at each irregular singular points; (3) appropriate connection matrices between canonical solutions at different singular points.

\subsection{Isomonodromic deformation of the linear system \eqref{eq:irregularsystem}}\label{sec:isomonodromydeformation}

In \cite{Fuchs1907,Schlesinger1912, JMMS1980, JMU1981I}, the authors viewed the Riemann-Hilbert-Birkhoff correspondence as a deformation problem and considered the deformations of the coefficients $A_i$'s of the linear system \eqref{eq:irregularsystem} with respect to the positions of singularities ($t_i$ parameters) and singular types ($u_i$ parameters) of the system \eqref{eq:irregularsystem} which keep the monodromy data of the the system \eqref{eq:irregularsystem} fixed. Let $\mathcal{T}$ denote the set of monodromy times $\{t_i, 1 \le i \le N; u_j, 1 \le j \le n\}$.

The isomonodromy equation of the linear system \eqref{eq:irregularsystem} with respect to $t_i, u_{\alpha}$'s are given by (c.f.\cite{JMMS1980})
\begin{align}
    & \frac{\partial F}{\partial t_i} = -\frac{A_i}{z - t_i}F,\label{eq:tideformation}\\
    & \frac{\partial F}{\partial u_{\alpha}} = (z E_{\alpha} + B_{\alpha})F,\label{eq:uideformation}
\end{align}
where
\[B_{\alpha} = -\sum \limits_{\beta \ne \alpha}\frac{E_{\alpha}A_{\infty}E_{\beta} + E_{\beta}A_{\infty}E_{\alpha}}{u_{\alpha} - u_{\beta}}, \qquad A_{\infty} = -\sum \limits_{i = 1}^N A_i.\]
The compatibility conditions of the systems \eqref{eq:irregularsystem}, \eqref{eq:tideformation} and \eqref{eq:uideformation} lead to
\begin{equation}\label{eq:zerocurvature}
    \begin{aligned}
        & \left[\frac{\partial}{\partial t_i} + \frac{A_i}{z - t_i}, \frac{\partial}{\partial z} - A(z)\right] = 0,\\
        & \left[\frac{\partial}{\partial u_{\alpha}} - \lambda E_{\alpha} - B_{\alpha}, \frac{\partial}{\partial z} - A(z)\right] = 0.
    \end{aligned}
\end{equation}
They reduce to the following equations for $A_i$'s (the JMMS equation \cite{JMMS1980})
\begin{equation}\label{eq:JMMS}
    \begin{aligned}
        & \frac{\partial A_i}{\partial t_i} = (1 - \delta_{ij})\frac{[A_i, A_j]}{t_i - t_j} + \delta_{ij}\left[U + \sum \limits_{k \ne i}\frac{A_k}{t_i - t_k}, A_j\right],\\
        & \frac{\partial A_j}{\partial u_{\alpha}} = [t_jE_{\alpha} + B_{\alpha}, A_j].
    \end{aligned}
\end{equation}

% System \eqref{eq:JMMS} is a Hamiltonian system
% \begin{equation}
%     \frac{\partial A_j}{\partial t_i} = \{A_j, H_i\}, \qquad \frac{\partial A_j}{u_{\alpha}} = \{A_j, K_{\alpha}\}
% \end{equation}
% under the standard Kostant-Kirillov Lie-Poisson bracket on $\mathfrak{gl}(n, \mathbb{C})$
% \[A_{i,\alpha\beta}, A_{j, \mu\nu} = \delta_{ij}(-\delta_{\beta\mu}A_{j, \alpha\nu} + \delta_{\nu\alpha}A_{j,\mu\beta}),\]
% with Hamiltonians
% \begin{align}
%     & H_i = \text{Tr}\left(UA_i + \sum \limits_{j \ne i}\frac{A_iA_j}{t_i - t_j}\right),\\
%     & K_{\alpha} = \text{Tr}\left(E_{\alpha}\sum \limits_{i = 1}^N t_iA_i + \sum \limits_{\beta \ne \alpha}\frac{E_{\alpha}E_{\infty}E_{\beta}A_{\infty}}{u_{\alpha} - u_{\beta}}\right).
% \end{align}

% There are two sets of invariants:
% \begin{enumerate}
%     \item The eigenvalues $\theta_{i\alpha}$ of $A_i$;
%     \item The diagonal elements of $A_{\infty}$.
% \end{enumerate}

\subsection{Harnad Duality}\label{sec:Harnadduality}

It was observed in \cite{Harnad1994} that system in the form \eqref{eq:irregularsystem} may have a dual system in the same form which is related to the original system by `Laplace transform' and exchanges the roles of $u_i$'s and $t_j$'s. 

Consider the case when rank $A_i = 1, 1 \le i \le N$ and we express $A_i$'s as
\[A_i = P_i^T Q_i,\]
where $\{(P_i, Q_i)\}_{i = 1, \dots, N}$ are pairs of nonzero $1 \times n$ row vectors.% with $\{Q_iP_i^T = B_i \in \mathbb{C}\}_{i = 1, \dots, N}$ constants related to the monodromy of \eqref{eq:irregularsystem} at $\{t_i\}_{i = 1, \dots, N}$. 

Let
\[P = \begin{pmatrix} P_1 \\ \vdots \\ P_N \end{pmatrix}, \qquad Q = \begin{pmatrix} Q_1 \\ \vdots \\ Q_N \end{pmatrix}\]
be $N \times n$ matrices with $P_i$'s and $Q_i$'s as row vectors. Note that
\[\sum \limits_{i = 1}^N \frac{A_i}{z - t_i} = P^T (z \text{Id}_N - V)^{-1} Q,\]
where $V = \text{diag}(t_1, \dots, t_N)$, then we can write the differential operator in \eqref{eq:irregularsystem} as
\begin{equation}\label{eq:irregularoperator1}
\frac{d}{dz} - P^T (z\text{Id}_N - V)^{-1}Q - U.
\end{equation}
The dual operator is given by
\begin{equation}\label{eq:irregularoperator2}
\frac{d}{dw} + Q (w\text{Id}_n - U)^{-1}P^T + V.
\end{equation}
Note that formally the linear system \eqref{eq:irregularsystem} is equivalent to
\begin{equation}\label{eq:irregularsystem2}
    \left(\frac{d}{dz} - U\right)X - P^TY = 0, \qquad (z\text{Id}_N - V)Y - QX = 0,
\end{equation}
where $X \in \mathbb{C}^n$ and $Y \in \mathbb{C}^N$. Then the linear system defined by operator \eqref{eq:irregularoperator2} is equivalent to the system 
\begin{equation}\label{eq:dualsystem}
    -\left(\frac{d}{dw} + V\right)Y - QX = 0, \qquad (w\text{Id}_n - U)X - P^TY = 0,
\end{equation}
which is obtained from a formal application of the Laplace transform to \eqref{eq:irregularsystem2} by replacing $\frac{d}{dz}$ and $z\text{Id}_N$ in \eqref{eq:irregularsystem2} by $w\text{Id}_n$ and $-\frac{d}{dw}$, respectively.

The following example connects two linear systems under Harnad duality whose isomonodromy deformation equations are both equivalent to PVI, and the explicit correspondences between their asymptotic behaviors at critical points and the corresponding monodromy formulas will be our main concern in the rest of the paper. 
\begin{example}\label{eg:DualityPVI}
Consider the linear system \eqref{eq:irregularsystem} with $N = 1, n = 3$, $U = {\rm diag}(0, x, 1), V = 0$ and rank $A_1 = 2$, so that the resulting equation is a special case of \eqref{introisoStokeseq1}. The corresponding differential operator can be put in the form of Equation \eqref{eq:irregularoperator1} by taking $P, Q$ to be $2 \times 3$ matrices. By Harnad duality, it's dual differential operator \eqref{eq:irregularoperator2} defines a $2 \times 2$ linear system \eqref{eq:irregularsystem} with only regular singularities at $0, x, 1, \infty$ whose isomonodromy deformation equation is known to be equivalent to PVI (See Section \ref{sec:regularsystem} and Section \ref{sec:irregularsystem} for more detailed discussions). 
% Taking $N = 2, n = 3, U = {\rm diag}(0, x, 1), V = 0$ and $P, Q$ to be $2 \times 3$ matrices, then Harnad duality induces isomorphism between $3 \times 3$ linear system \eqref{eq:irregularsystem} with one regular singular point at $0$ and one irregular singular point of Poincar\'e rank $1$ at $\infty$ and $2 \times 2$ system \eqref{eq:irregularsystem} with only regular singularities at $0, x, 1, \infty$(See Section \ref{sec:regularsystem} and Section \ref{sec:irregularsystem} for more detailed discussions). 
\end{example}

\section{ Painlev\'e VI equation and isomonodromic deformation }

% \subsection{PVI and isomonodromy deformation}\label{sec:mathchingmonodromy}

% We first collect some facts about PVI from the isomonodromic deformation of regular system \eqref{eq:2by2system}.

\subsection{PVI from isomonodromy deformation of Fuchsian system and Jimbo's formula}\label{sec:regularsystem}
First consider the following case in the setting of Section \ref{sec:isomonodromydeformation}: taking $N = 3, n = 2, V = 0$,  the JMMS equation of the $2 \times 2$ linear system \eqref{eq:irregularsystem} is
    \begin{align}
        & \frac{\partial F}{\partial z} = A(z)F, \qquad A(z) = \frac{A_1}{z-u_1}+ \frac{A_2}{z - u_2} + \frac{A_3}{z - u_3}   = \begin{pmatrix} \label{eq:2by2system}
        a_{11} & a_{12} \\ a_{21} & a_{22}
        \end{pmatrix},\\
        & \frac{\partial F(z)}{\partial u_i} = -\frac{A_i}{z - u_i}F(z). \label{eq:2by2auxiliary}
    \end{align}
    Here  we assume $A_i\in SL(2,\mathbb{C})$ satisfy the following conditions:
    \begin{align} \label{eq:eigen of A in fuchsian sys}
        \text{Eigenvalues of }\ A_i = \pm \frac{\theta_i}{2}, \ \quad \sum \limits_{i = 1}^3 A_i = \begin{pmatrix} 
-\theta_{\infty}\slash 2 & 0 \\ 0 & \theta_{\infty} \slash 2
\end{pmatrix}. 
    \end{align}
\begin{remark}
Linear system \eqref{eq:2by2system} is slightly different from the one given in Example \ref{eg:DualityPVI} where $A_i$'s all have rank $1$. But they are equivalent up to a shift by scalar matrices on $A_i$'s, and both of them induce Painlev\'e VI (c.f. \cite{Jimbo-Miwa1981II, Mahoux1999}). 
\end{remark}
    
    The compatibility conditions for linear systems \eqref{eq:2by2system} and \eqref{eq:2by2auxiliary} give the  following Schlesinger system
    \begin{align}\label{eq:Schlesinger sys}
        \frac{\partial}{\partial u_j}A_i = \frac{[A_i,A_j]}{u_i-u_j}, \ \quad \frac{\partial}{\partial u_i}A_i = -\sum_{j\neq i} \frac{[A_i,A_j]}{u_i-u_j}.
    \end{align}
    The solutions of Schlesinger system, up to conjugation with a invertible
    diagonal matrix, are  in one-to-one correspondence with solutions of Painlev\'e VI equation \eqref{eq:PVI} with $x= (u_2-u_1)\slash(u_3-u_1)$, $y = (q-u_1)\slash(u_3-u_1)$, where $q$ is the root of $a_{12}$ in \eqref{eq:2by2system},  considered up to some symmetries (c.f. \cite{JMU1981I},\cite{Mazzocco2001rational}).
% Let eigenvalues of $A_i(t)$ in \eqref{eq:2by2system} be $\pm \theta_i \slash 2$, and assume
% \[\sum \limits_{i = 1}^3 A_i = \begin{pmatrix} 
% -\theta_{\infty}\slash 2 & 0 \\ 0 & \theta_{\infty} \slash 2
% \end{pmatrix}.\]
%
% \begin{align} \label{eq:eigen of Ai}
%     & \text{Eigenvalues of }\ A_i(u) \ \text{are}\ \pm \theta_i ;\\ \label{eq:normalization of A inf}
%     &\sum_{i=1}^3 A_i = \left(
%     \begin{matrix}
%         -\theta_{\infty}/2  &   0 \\
%         0         &     \theta_{\infty}/2  
%     \end{matrix}
%     \right).
% \end{align}
% Then the generic condition \eqref{eq:genericconditions} requires
% \begin{align}
%     \theta_1, \theta_2, \theta_3, \theta_\infty \notin \mathbb{Z}.
% \end{align}

Let us now introduce Jimbo's formula to determine the critical behaviour of PVI using isomonodromy deformation of Fuchsian system. %The local behavior of fundamental solution matrix of \eqref{eq:2by2system} then reads as follows 
% \begin{align}
%     \begin{split}
%         F(z,u) &= \left(G^{(1)}(u)+O(z-u_1)\right){\left(\begin{matrix}
%             (z-u_1)^{\theta_1/2} & 0  \\
%              0 & (z-u_1)^{-\theta_1/2}
%         \end{matrix}\right)}C^{(1)}  \qquad (z\rightarrow u_1) \\
%         &= \left(G^{(2)}(u)+O(z-u_2)\right){\left(\begin{matrix}
%             (z-u_2)^{\theta_2/2} & 0  \\
%              0 & (z-u_2)^{-\theta_2/2}
%         \end{matrix}\right)}C^{(2)} \qquad  (z\rightarrow u_2) \\
%         &= \left(G^{(3)}(u)+O(z-u_3)\right){\left(\begin{matrix}
%             (z-u_3)^{\theta_3/2} & 0  \\
%              0 & (z-u_3)^{-\theta_3/2}
%         \end{matrix}\right)}C^{(3)} \qquad  (z\rightarrow u_3) \\
%         &= \left(1+O(z^{-1})\right){\left(\begin{matrix}
%             z^{-\theta_\infty/2} & 0  \\
%              0 & z^{\theta_\infty/2}
%         \end{matrix}\right)} \qquad \qquad \qquad \qquad  \qquad \qquad \quad \ (z\rightarrow \infty)
%     \end{split}
% \end{align}
% where $C^{(k)},\ k=1,2,3$ are constant matrices. 
Let $M_1,M_2,M_3 \in SL(2,\mathbb{C})$ be the monodromy matrices of fundamental solution of system \eqref{eq:2by2system} %at $z =u_1, u_2, u_3 $ respectively, i.e. 
along a given simple loop surrounding $z = u_1, u_2, u_3$, %i.e.
% \begin{align}\label{theta}
%     M_{k} =  C^{(k)-1}\left(\begin{matrix}
%         e^{\pi i\theta_k} & 0  \\
%         0  &  e^{-\pi i \theta_k}
%     \end{matrix}\right) 
%      C^{(k)}, 
%      % \quad C^{(k)} \ \text{is a constant matrix},
% \end{align}
and define $M_\infty \in SL(2,\mathbb{C})$ by
\begin{align}\label{eq:monodromy relation}
    M_\infty M_3 M_2 M_1 = {\rm Id_2}.
\end{align}
It provides us seven parameters 
\begin{align}\label{eq:def of pij}
    p_{ij} &= \mathrm{tr} (M_iM_j) , \ 1\leq i < j \leq 3, \\ \label{ptheta}
    p_k & = 2\mathrm{cos}(\pi \theta_k),\ k=1,2,3, \infty,
\end{align}
with the following Jimbo-Fricke cubic relation as a constraint \cite{Jimbo1982, Boalch2005}
\begin{align}\label{eq:relation among pij}
    \begin{split}
        0 &= p_{13}p_{23}p_{12}+p_{12}^2+p_{23}^2+p_{13}^2 \\
        &\quad -(p_1p_3+p_2p_{\infty})p_{13}-(p_3p_2+p_1p_{\infty})p_{23}-(p_2p_1+p_3p_{\infty})p_{12} \\
        &\quad +p_1^2+p_2^2+p_3^2+p_\infty^2+p_1p_2p_3p_\infty-4.
    \end{split}
\end{align}
These parameters characterize the three monodromy matrices (c.f. \cite{Jimbo1982}), and the identity \eqref{eq:relation among pij} is from \eqref{eq:monodromy relation}.

It is proved that for each branch of solutions of Schelesinger system \eqref{eq:Schlesinger sys}, the triples of monodromy matrices $M_1,M_2,M_3$ are uniquely determined (unless $M_k=\pm 1,\ k=1,2,3,\infty$), up to conjugation with the same invertible constant matrix (c.f. \cite{JMU1981I}). Thus the branches of solutions of Painlev\'e VI equation are determined by the triple $[(M_1,M_2,M_3)]$. Jimbo then gave the critical behaviour of generic solutions of PVI, using the parameters of monodromy matrices. Here we take the formulation in \cite{Boalch2005}, which corrects a sign in \cite{Jimbo1982}.
\begin{theorem}[\cite{Jimbo1982}]\label{thm Jimbo's formua for leading term of PVI}
    Suppose we have four matrices $M_j \in SL_2(\mathbb{C})$, for $j=1,2,3,\infty$ satisfying
    
    $(a)$ $M_{\infty}M_{3}M_{2}M_{1}=1$, 

    $(b)$ $M_j$ has eigenvalues $\{{\rm exp}(\pm \pi i\theta_j)\}$ with $\theta_j\notin \mathbb{Z}$,

    $(c)$ $\mathrm{Tr}(M_1M_2)=2\mathrm{cos}(\pi\sigma)$ for some non-zero $\sigma\in \mathbb{C}$ with $0\leq Re(\sigma) < 1$,

    $(d)$ none of the following eight numbers is an even integer:
    $$\theta_1\pm \theta_2\pm \sigma,\ \theta_1\pm\theta_2\mp\sigma,\ \theta_\infty\pm\theta_3\pm\sigma,\ \theta_\infty\pm\theta_3\mp\sigma.$$

    Then the leading term in the asymptotic expansion at zero of corresponding PVI solution $y(x)$ on the branch corresponding to $[(M_1,M_2,M_3)]$ is 
    \begin{align}\label{eq:expansion of J}
        \frac{(\theta_1+\theta_2+\sigma)(-\theta_1+\theta_2+\sigma)(\theta_\infty+\theta_3+\sigma)}{4\sigma^2(\theta_\infty+\theta_3-\sigma)\hat{s}} x^{1-\sigma},
    \end{align}
    where 
    $$\hat{s} = c\times s,\quad s = (a+b)/d, $$
    and 
    \begin{align}\label{eq:a,b,c,d in J}
    \begin{split}
        a &= e^{\pi i\sigma}(i\sin{(\pi\sigma)}\cos{(\pi\sigma_{23})}-\cos{(\pi\theta_2)}\cos{(\pi\theta_\infty)}-\cos{(\pi\theta_1)}\cos{(\pi\theta_3})), \\
        b &=i\sin{(\pi\sigma)}\cos{(\pi\sigma_{13})}+\cos{(\pi\theta_2)}\cos{(\pi\theta_3)}+\cos{(\pi\theta_\infty)}\cos{(\pi\theta_1)},\\
        d&= 4\sin{(\frac{1}{2}\pi(\theta_1+\theta_2-\sigma))}\sin{(\frac{1}{2}\pi(\theta_1-\theta_2+\sigma))}\sin{(\frac{1}{2}\pi(\theta_\infty+\theta_3-\sigma))}\sin{(\frac{1}{2}\pi(\theta_\infty-\theta_3+\sigma))},\\
         c&=\frac{(\Gamma(1-\sigma))^2\hat{\Gamma}(\theta_1+\theta_2+\sigma)\hat{\Gamma}(-\theta_1+\theta_2+\sigma)\hat{\Gamma}(\theta_\infty+\theta_3+\sigma)\hat{\Gamma}(-\theta_\infty+\theta_3+\sigma)}{(\Gamma(1+\sigma))^2\hat{\Gamma}(\theta_1+\theta_2-\sigma)\hat{\Gamma}(-\theta_1+\theta_2-\sigma)\hat{\Gamma}(\theta_\infty+\theta_3-\sigma)\hat{\Gamma}(-\theta_\infty+\theta_3-\sigma)}.
    \end{split}
    \end{align}
    Here $\Gamma(x)$ is the usual gamma function and $\hat{\Gamma}(x) := \Gamma(1+\frac{1}{2}x)$, and $\sigma_{jk}$ is determined by ${\rm Tr}(M_jM_k)) = 2\cos{(\pi\sigma_{jk})}$, with $0\leq {\rm Re}(\sigma_{jk})\leq 1
    $, so $\sigma = \sigma_{12}$.
 \end{theorem}

According to Theorem \ref{thm Jimbo's formua for leading term of PVI}, we see that the parameters $J$ and $\sigma$
in the leading asymptotics of $y(x) = Jx^{1-\sigma}(1+O(x^\epsilon))$ are related to the monodromy data by \eqref{eq:expansion of J} and $p_{12}=\text{Tr}(M_1M_2)=2\cos{(\pi\sigma)}$.

% Next we interpret the critical behavior in Theorem \ref{thm Jimbo's formua for leading term of PVI} from isomonodromy deformation of irregular system \eqref{eq:3by3system}.

\subsection{Isomonodromy deformation of a linear system with irregular singularity}\label{sec:irregularsystem}
   
  Next we consider the dual case in the setting of Section \ref{sec:isomonodromydeformation} and \ref{sec:Harnadduality}: taking $N = 1, n = 3$, the JMMS equation of the dual $3 \times 3$ linear system is
    \begin{align}
        & \frac{\partial F}{\partial z} = \left(U + \frac{\Phi(u)}{z}\right)F, \qquad U = \text{diag}(u_1, u_2, u_3), 
                \label{eq:3by3system}\\
        & \frac{\partial F}{\partial u_\alpha} = (zE_{\alpha} + B_{\alpha})F, \quad B_\alpha = \mathrm{ad}_{D_\alpha} \mathrm{ad}_{E_\alpha}\Phi, \ \alpha=1,2,3, 
    \end{align}
    where we assume $u=(u_1,u_2,u_3)\notin \Delta_{\mathbb{C}^3} = \bigcup \limits_{i \ne j} \{u \in \mathbb{C}^3 \ \vert\ u_i - u_j = 0\}$, and 
\begin{align*}
   D_1 = {\rm diag}\left(0,\frac{1}{u_2-u_1},\frac{1}{u_3-u_1}\right),\    D_2 = {\rm diag}\left(\frac{1}{u_1-u_2},0,\frac{1}{u_3-u_2}\right),\   D_3 = {\rm diag}\left(\frac{1}{u_1-u_3},\frac{1}{u_2-u_3},0\right).
    \end{align*}
The compatibility condition gives the following isomonodromy equation (the special case of \eqref{isoeq} for $n=3$)
\begin{equation}\label{eq:isomonodromy3}
        \frac{\partial}{\partial u_\alpha}\Phi(u) = [\text{ad}_{D_\alpha}\text{ad}_{E_\alpha}\Phi(u), \Phi(u)],\ \alpha=1,2,3.
    \end{equation}

Now let us recall the Stokes matrices of system \eqref{eq:3by3system}. 
For our purpose, we only need to consider the system \eqref{eq:3by3system} with the irregular data $U= {\rm diag}(u_1,..., u_n)$ with purely imaginary parameters $u_1,...,u_n$ satisfying ${\rm Im}(u_1)<{\rm Im}(u_2)<\cdots <{\rm Im}(u_n)$. In this case, the Stokes sectors of \eqref{eq:3by3system}
are the right/left half planes ${\rm Sect}_\pm=\{z\in\mathbb{C}~|~ \pm{\rm Re}(z)>0\}$. The following result of the canonical solutions via the Laplace-Borel transforms is standard (c.f. \cite{Balser, LR}). 

Let us choose the branch of ${\rm log}(z)$, which is real on the positive real axis, with a cut along the nonnegative imaginary axis $i\mathbb{R}_{\ge 0}$. Then by convention, ${\rm log}(z)$ has imaginary part $-\pi$ on the negative real axis in ${\rm Sect}_-$.
\begin{theorem}[c.f. \cite{Balser, LR}]
    \label{uniformresum}
On ${\rm Sect}_\pm$ there
is a unique (therefore canonical) fundamental solution $F_\pm:{\rm Sect}_\pm\to {\rm GL}_n(\mathbb{C})$ of equation \eqref{eq:3by3system} such that $F_+\cdot e^{- uz}\cdot z^{-\delta \Phi}$ and $F_-\cdot e^{- uz}\cdot z^{-\delta \Phi}$ can be analytically continued to sectors $\{z\in\mathbb{C}~|-\pi<{\rm arg}(z)<\pi\}$ and $\{z\in\mathbb{C}~|-2\pi<{\rm arg}(z)<0\}$ respectively, and
\begin{align*}
\lim_{z\rightarrow\infty}F_+(z;u)\cdot e^{-uz}\cdot z^{-\delta \Phi}&={\rm Id}_n, \ \ \ \text{as} \ \ -\pi<{\rm arg}(z)< \pi,
\\
\lim_{z\rightarrow\infty}F_-(z;u)\cdot e^{-uz}\cdot z^{-\delta \Phi}&={\rm Id}_n, \ \ \ \text{as} \ \ -2\pi< {\rm arg}(z)<0.
\end{align*}
Here ${\rm Id}_n$ is the rank n identity matrix, and $\delta \Phi$ is the diagonal part of $\Phi$. The solutions $F_\pm$ are called the canonical solutions in ${\rm Sect}_\pm$.
\end{theorem}

\begin{definition} 
\label{defiStokes}
The {\it Stokes matrices} of the system \eqref{eq:3by3system} (with respect
to ${\rm Sect}_+$ and the branch of ${\rm log}(z)$) are the elements $S_\pm(u,\Phi)\in {\rm GL}(n)$ determined by
\begin{eqnarray}
F_+(z;u)&=&F_-(z;u)\cdot e^{i\pi\delta \Phi} S_+(u,\Phi),\\
F_-(ze^{-2i\pi };u)&=&F_+(z;u)\cdot S_-(u,\Phi)e^{-i\pi\delta\Phi},
\end{eqnarray}
where the first (resp. second) identity is understood to hold in ${\rm Sect}_-$
(resp. ${\rm Sect}_+$) after $ F_+$ (resp. $F_-$)
has been analytically continued clockwise. 
\end{definition} 
The prescribed asymptotics of $F_\pm(z;u)$ at $z=0$, as well as the identities in Definition \ref{defiStokes}, ensures that the Stokes matrices $S_+(u,\Phi)$ and $S_-(u,\Phi)$ are upper and lower triangular matrices respectively, with diagonal part are both $e^{-i\pi\delta \Phi}$ (c.f. \cite[Chapter 9.1]{Balser}).

Analogous to Theorem \ref{thm Jimbo's formua for leading term of PVI} for PVI, Theorem \ref{Thm:termwise} and \ref{thm: introcatformula} for \eqref{eq:isomonodromy3} characterize the asymptotic behavior of $\Phi(u)$ at a critical point and the monodromy problem of the associated linear system \eqref{eq:3by3system}.

\subsection{The correspondence between the monodromy and solutions under the duality}\label{sec:monodromyduality}

Under the Harnad duality \cite{Harnad1994}, both the monodromy of the linear equations \eqref{eq:2by2system} and \eqref{eq:3by3system} and the solutions of the equation \eqref{eq:Schlesinger sys} (therefore \eqref{eq:PVI}) and \eqref{eq:isomonodromy3} have a correspondence.

First note that if $\Phi(u)$ is a solution of \eqref{eq:isomonodromy3}, so is the translation $\Phi' = \Phi - \lambda \cdot {\rm Id}_3$, with $\lambda$ a constant. Thus in the following, we can assume that the solutions of \eqref{eq:isomonodromy3} satisfy \eqref{eq:restr of diag of phi}. Note that the condition \eqref{eq:restr of eigen of phi} is equivalent to $\theta_\infty \neq 0,\  \pm (\theta_1+\theta_2+\theta_3)$.
% \begin{align}

%     \theta_\infty \neq 0,\  \pm (\theta_1+\theta_2+\theta_3).
% \end{align}

%{\color{red} stress the correspondence between the parameters $\theta_1, \theta_2, \theta_3, \theta_4$.}

Then under the duality between the $3\times 3$ system \eqref{eq:3by3system} and the $2\times 2$ system \eqref{eq:2by2system}, the parameters $\theta_1,\theta_2,\theta_3,\theta_\infty$ in \eqref{eq:eigen of A in fuchsian sys} coincide with the ones in \eqref{eq:restr of diag of phi}-\eqref{eq:restr of eigen of phi}. And we have the following correspondence between the associated monodromy data.

\begin{theorem}\cite{Balser-Jurkat-Lutz1981, Boalch2005, Degano-Guzzetti2023}\label{thm: relation of pij and S+-}
The monodromy parameters $p_{ij}$ of the Fuchsian system \eqref{eq:2by2system} are expressed in terms of the entries of the Stokes matrices $S_{\pm}(u,\Phi(u))$ of the dual system \eqref{eq:3by3system} by
\begin{align}\label{eq:relation of pij and S+-}
        p_{ij} &= 2\cos{\pi(\theta_i-\theta_j)}-(S_+)_{ij}(S_-^{-1})_{ji}.
        % \\ \label{eq:rela of pk and S+-}
        % p_{k} &= (S_+)_{kk}+(S_-^{-1})_{kk}.
    \end{align}
Here $\theta_k$'s are defined in \eqref{eq:eigen of A in fuchsian sys}, and $p_{ij}$'s are defined in \eqref{eq:def of pij}.
\end{theorem}

\begin{remark}
    Our definition of $S_+$, $S_-$ are different from the Stokes matrices $S_1$, $S_2$ defined in \cite{Degano-Guzzetti2023}. But after permutation, they are simply related by 
    \begin{align}
        S_1 &= e^{\mathrm{i}\pi\delta\Phi}\cdot S_+;\\
        S_2 &= S_-\cdot e^{\mathrm{i}\pi\delta\Phi}.
    \end{align}
\end{remark}
\begin{remark}
The quantization of monodromy spaces of the Fuchsian system \eqref{eq:2by2system} and the dual system \eqref{eq:3by3system} are studied in \cite{CMR, Marzzocco2023generalized}, and the quantum analog of Theorem \ref{thm: relation of pij and S+-} is given in \cite{Marzzocco2023generalized}, with a relation to the generalized double affine Hecke algebras and quantum cluster varieties \cite{CMR}.
% \bibitem{CMR}
% L. O. Chekhov, M. Mazzocco, and V. N. Rubtsov, Quantised Painlevé monodromy manifolds, Sklyanin and Calabi–Yau algebras. Advances in Mathematics, 376, 2

% \bibitem{MM}
% D. Martello and M. Mazzocco,
% Generalized double affine Hecke algebras,
% their representations and higher Teichm\"uller
% theory, arXiv.2307.06803.
\end{remark}
Furthermore, the solutions of PVI \eqref{eq:PVI} (equivalently the equation \eqref{eq:Schlesinger sys}) are in one-to-one correspondence to the solutions of isomonodromy equations \eqref{eq:isomonodromy3} , up to conjugation by a nonsingular diagonal matrix. More precisely, we have the following theorem from Mazzocco \cite{Mazzocco2002}, Degano and Guzzetti \cite{Degano-Guzzetti2023}.

\begin{theorem}\cite{Degano-Guzzetti2023}\label{thm: Expression of Omega in solution of PVI}
The $3\times 3$ isomonodromy equation \eqref{eq:isomonodromy3} with the constraints \eqref{eq:restr of diag of phi} and \eqref{eq:restr of eigen of phi}, is equivalent to
PVI \eqref{eq:PVI}.
Furthermore, there is a one-to-one correspondence between transcendents $y(x)$ and equivalence classes
\[\Big\{K^0 \cdot \Phi(u) \cdot (K^0)^{-1}, \ K^0 = {\rm diag}(k_1^0,k_2^0,1), \  k_1^0,k_2^0\in \mathbb{C}\setminus \{0\},\Big\} \]
of solutions of \eqref{eq:isomonodromy3}. And under the correspondence the following explicit formula holds: let us introduce
    \begin{align} \label{eq:defini of Omega}
        x = \frac{u_2-u_1}{u_3-u_1},\quad  \Omega(x) = (u_3-u_1)^{\delta\Phi}\Phi(u)(u_3-u_1)^{-\delta\Phi},
    \end{align}
then
    \begin{align*}
        & \Omega_{12} = \frac{k_1(x)}{k_2(x)} \cdot \frac{f(\theta_1, \theta_2, \theta_3, \theta_{\infty}; x, y)}{2(1 - x)y}, \ \ \ \Omega_{21} = \frac{k_2(x)}{k_1(x)} \cdot \frac{f(-\theta_1, -\theta_2, \theta_3, \theta_{\infty}; x, y)}{2(y - x)},\\
        & \Omega_{13} = k_1(x) \cdot \frac{f(\theta_1, \theta_1 - \theta_3 - \theta_{\infty}, \theta_3, \theta_{\infty}; x, y)}{2(x - 1)y}, \ \  \ \Omega_{31} = \frac{1}{k_1(x)} \cdot \frac{f(-\theta_1, -\theta_1 + \theta_3 - \theta_{\infty}, \theta_3, \theta_{\infty}; x, y)}{2x(y-1)},\\
        & \Omega_{23} = k_2(x) \cdot \frac{f(-\theta_2 + \theta_3 + \theta_{\infty}, -\theta_2, \theta_3, \theta_{\infty}; x, y)}{2(x - y)}, \ \ \ \Omega_{32} = \frac{1}{k_2(x)} \cdot \frac{f(\theta_2 - \theta_3 + \theta_{\infty}, \theta_2, \theta_3, \theta_{\infty}; x, y)}{2x(1 - y)}.
    \end{align*}
Here
\[f(\theta_1, \theta_2, \theta_3, \theta_{\infty}; x, y) := (x - x^2)\frac{dy}{dx} + (1 - \theta_{\infty})y^2 + [(\theta_2 + \theta_{\infty})x + \theta_1 - \theta_2 - 1]y - \theta_1 x,\]
and
    \begin{align*}
        & l_1(x) := \frac{f(\theta_1, \theta_2, \theta_3, \theta_1 + \theta_3 - \theta_2; x, y)}{2x(1-x)(1-y)y},\\
        & l_2(x) := \frac{f(-\theta_1, -\theta_2, \theta_3, \theta_3 + \theta_2 - \theta_1; x, y)}{2x(y-1)(x-y)} + \frac{\theta_2 - \theta_3}{x - 1},\\
        & k_i(x) = k_i^0\exp\left({\int^xl_i(\xi)d\xi}\right), \quad k_i^0 \in \mathbb{C} \backslash \{0\}, \qquad i = 1, 2.
    \end{align*}
\end{theorem}

\begin{remark}
The function $f(\theta_1, \theta_2, \theta_3, \theta_{\infty}; x, y)$ has the following explanation from the Hamiltonian realization of $PVI$: let $z(x)$ in \cite[(2.51)]{Degano-Guzzetti2023} be the momentum coordinate conjugate to $y(x)$ (c.f. \cite[(3.72)]{Harnad1994}), then \[z(x) = -\frac{f(\theta_1, -\theta_2, \theta_3, \theta_{1}+\theta_3 +\theta_2; x, y)}{2y(y-x)(y-1)}.\]
\end{remark}

\section{The correspondence of the boundary values and a new interpretation of Jimbo's formula}\label{lastsec}
\subsection{The proof of Theorem \ref{mainthm}}\label{p1.4}
Now we proceed to give a proof of Theorem \ref{mainthm} which is divided into several steps.
Suppose we are given a generic $3\times 3$ matrix solution $\Phi(u)$ of \eqref{eq:isomonodromy3}, 
then as a direct consequence of Theorem \ref{Thm:termwise} we have
\begin{cor}\label{lemma:limit of Omega}
     For the $3\times 3$ matrix function $\Omega(x)$ defined from $\Phi(u)$ by \eqref{eq:defini of Omega}, there exits a constant $3\times 3$ matrix $\Phi_0$, such that
     \begin{align}
    A(x)&:=\delta_2 (x^{\delta\Phi}\Omega(x)x^{-\delta\Phi})\rightarrow \delta_2\Phi_0,\quad \text{as } x\rightarrow 0, \label{eq:delta Phi0} \\
    B(x)&:=x^{-\delta_2\Phi_0} x^{\delta\Phi} \Omega(x)x^{-\delta\Phi}x^{\delta_2\Phi_0} \rightarrow \Phi_0, \quad \text{as }  x\rightarrow 0. \label{eq:Phi0}
\end{align}
By definition, this $\Phi_0$ is the boundary value of $\Phi(u)$ as $\frac{u_3-u_2}{u_2-u_1}\rightarrow \infty$ (equivalently as $x=\frac{u_2-u_1}{u_3-u_1}\rightarrow 0$). 
\end{cor}

In the following, let us compute explicitly $\Phi_0$ from the formula of $\Omega(x)$ given in Theorem \ref{thm: Expression of Omega in solution of PVI}. In the computation, let us assume ${\rm Re}(\sigma)>0$, the formula in Theorem \ref{mainthm} can be continuously extended to the case ${\rm Re}(\sigma)=0$ and $\sigma\ne 0$ (see Remark \ref{rem:extend sigma}).

%It should be pointed out that in general we cannot compute the limit of $F(x,y(x),y'(x))$, when only knowing the leading term of $y(x)$. The proof of theorem \ref{mainthm} is based on the special structure of $\Omega(x)$ in theorem \ref{thm: Expression of Omega in solution of PVI}, which can be seen clearly in the section \ref{sec: limit of B(x)}.

\subsubsection{Asymptotic expansions of $k_1(x)$ and $k_2(x)$}
Recall that the parameters $\theta_1,\theta_2,\theta_3,\theta_\infty$ in theorem \ref{thm: Expression of Omega in solution of PVI} are given in \eqref{eq:restr of diag of phi}- \eqref{eq:restr of eigen of phi}, and $\sigma, J$ are parameters in the leading asymptotics of $y(x)$ as in \eqref{VIasy}.
\begin{lemma}\label{lemma k1,k2}
    
    $k_1(x)$ and $k_2(x)$ in Theorem \ref{thm: Expression of Omega in solution of PVI} have the following asymptotic expansions as $x\rightarrow 0$
    \begin{align}
        k_1(x) &= k_1^0\cdot x^{\frac{\theta_1-\theta_2-\sigma}{2}}(1+f_0(x))(1-\frac{\theta_1}{2\sigma J}x^\sigma) + o(x^{\frac{\theta_1-\theta_2+\sigma}{2}}), \label{k1(x)} \\
        k_2(x) &= k_2^0 \cdot x^{\frac{-\theta_1+\theta_2-\sigma}{2}} (1+f_0(x))(1+\frac{\theta_2-\sigma}{2\sigma J}x^{\sigma})+ o(x^{\frac{-\theta_1+\theta_2+\sigma}{2}}),\label{k2(x)}
    \end{align}

    where $f_0(x)={\rm exp}\left(\int_0^{x} \left(\frac{1}{2y(1-y)}\frac{dy}{d\xi}-\frac{1-\sigma}{2\xi}+ \frac{-\theta_3 y}{2\xi(1-y)}\right)\ d\xi\right)-1=o(1)$, as $x\rightarrow 0$.
    
\end{lemma}

\begin{proof}
    Recall that we have
        \begin{align*}
        & f(\theta_1, \theta_2, \theta_3, \theta_{\infty}; x, y) = (x - x^2)\frac{dy}{dx} + (1 - \theta_{\infty})y^2 + [(\theta_2 + \theta_{\infty})x + \theta_1 - \theta_2 - 1]y - \theta_1 x,\\
        & l_1(x) := \frac{f(\theta_1, \theta_2, \theta_3, \theta_1 + \theta_3 - \theta_2; x, y)}{2x(1-x)(1-y)y},\\
        & l_2(x) := \frac{f(-\theta_1, -\theta_2, \theta_3, \theta_3 + \theta_2 - \theta_1; x, y)}{2x(y-1)(x-y)} + \frac{\theta_2 - \theta_3}{x - 1},\\
        & k_i(x) = k_i^0\exp\left(\int^xl_i(\xi)d\xi\right), \quad k_i^0 \in \mathbb{C} \backslash \{0\}, \qquad i = 1, 2.
    \end{align*}
    
    To obtain asymptotics for $k_1(x)$, from estimation for each term of $f(\theta_1, \theta_2, \theta_3, \theta_{\infty}; x, y)$, we first note that
    \begin{equation}
        \begin{split}
            l_1(x) &= \frac{1}{2y(1-y)}\frac{dy}{dx}+ 
           \left( \frac{(\theta_2-\theta_1-\theta_3+1)y}{2x(1-y)}+O(x^{1-\sigma}) \right) \\
            &\quad +\left(\frac{(\theta_1-\theta_2-1)y}{2x(1-y)} +\frac{\theta_1-\theta_2-1}{2x}+ O(1)\right) +\left(-\frac{\theta_1}{2J}x^{\sigma-1} + o(x^{\sigma-1})\right)\\
            &=\left(\frac{1}{2y(1-y)}\frac{dy}{dx}-\frac{1-\sigma}{2x}- \frac{\theta_3 y}{2x(1-y)}\right) +  \frac{\theta_1-\theta_2-\sigma}{2x}-\frac{\theta_1}{2J}x^{\sigma-1} + o(x^{\sigma-1}) ,
        \end{split}
    \end{equation}
where $\frac{1}{2y(1-y)}\frac{dy}{dx}-\frac{1-\sigma}{2x}- \frac{\theta_3 y}{2x(1-y)}=o(x^{-1})$, as $x\rightarrow 0$. In the above identities, we used the following types of estimates
\begin{align*}
    & \frac{(\theta_2-\theta_1-\theta_3+1)y^2}{2x(x-1)(y-1)y} = \frac{(\theta_2-\theta_1-\theta_3+1)y}{2x(1-y)}+O(x^{1-\sigma}),\\
     % \frac{(\theta_2-\theta_1-\theta_3+1)y^2}{2x(x-1)(y-1)y} = \frac{(\theta_2-\theta_1-\theta_3+1)y}{2x(1-y)}+O(x^{1-\sigma}),\\
   & \frac{((\theta_1+\theta_3)x+\theta_1-\theta_2-1)y}{2x(x-1)(y-1)y} = \frac{(\theta_1-\theta_2-1)y}{2x(1-y)} +\frac{\theta_1-\theta_2-1}{2x}+ O(1), \\
    & \frac{-\theta_1x}{2x(x-1)(y-1)y} = -\frac{\theta_1}{2J}x^{\sigma-1} + o(x^{\sigma-1}).
\end{align*}
    Thus,
    \begin{align}
       \begin{split}
          k_1(x) &= k_1^0\cdot {\rm exp}\left(\int^{x} l_1(\xi)d\xi\right)\\
        &= k_1^0\cdot x^{\frac{\theta_1-\theta_2-\sigma}{2}}
        {\rm exp}\left(\int^{x} \frac{1}{2y(1-y)}\frac{dy}{d\xi}-\frac{1-\sigma}{2\xi}+ \frac{-\theta_3 y}{2\xi(1-y)}\right)\rm exp\left(-\frac{\theta_1}{2\sigma J}x^{\sigma}\right){\rm exp}\left(o(x^\sigma)\right)\\
        &= k_1^0\cdot x^{\frac{\theta_1-\theta_2-\sigma}{2}}
        (1+f_0(x))(1-\frac{\theta_1}{2\sigma J}x^\sigma) + o(x^{\frac{\theta_1-\theta_2+\sigma}{2}}),
       \end{split}
    \end{align}
where $f_0(x)={\rm exp}\left(\int_0^{x} \left(\frac{1}{2y(1-y)}\frac{dy}{d\xi}-\frac{1-\sigma}{2\xi}- \frac{\theta_3 y}{2\xi(1-y)}\right)\ d\xi\right)-1=o(1)$, as $x\rightarrow 0$, and here $k_1^0$ may represent different constants of integration in different lines.
    
Similarly, for $k_2(x)$, we have
    \begin{align}
        \begin{split}
            l_2(x)&=\left(\frac{1}{2y(1-y)}\frac{dy}{dx}+\frac{1}{2y(1-y)}\cdot\frac{x}{y}\frac{dy}{dx}+o(x^{\sigma-1})\right)+ \left(\frac{(\theta_1-\theta_2-\theta_3+1)y}{2x(1-y)}+O(1)\right)\\
            &\quad +\left(\frac{-\theta_1+\theta_2-1}{2x(1-y)} +\frac{-\theta_1+\theta_2-1}{2x(1-y)}\cdot\frac{x}{y}+o(x^{\sigma-1})\right)+\left(\frac{\theta_1}{2J}x^{\sigma-1}+o(x^{\sigma-1})\right)\nonumber \\ 
            & = \left(\frac{1}{2y(1-y)}\frac{dy}{dx}+\frac{1-\sigma}{2J}x^{\sigma-1}+o(x^{\sigma-1})\right)+ \left(\frac{(\theta_1-\theta_2-\theta_3+1)y}{2x(1-y)}+O(1)\right)\\
            &\quad +\left(\frac{-\theta_1+\theta_2-1}{2x}+\frac{(-\theta_1+\theta_2-1)y}{2x(1-y)} +\frac{-\theta_1+\theta_2-1}{2J}x^{\sigma-1}+o(x^{\sigma-1})\right)+\left(\frac{\theta_1}{2J}x^{\sigma-1}+o(x^{\sigma-1})\right)\nonumber             \nonumber \\
            &=\left(\frac{1}{2y(1-y)}\frac{dy}{dx}-\frac{1-\sigma}{2x}+ \frac{-\theta_3 y}{2x(1-y)}\right) +  \frac{-\theta_1+\theta_2-\sigma}{2x}+\frac{\theta_2-\sigma}{2J}x^{\sigma-1} + o(x^{\sigma-1}).
        \end{split}
    \end{align}
    Here we used
    \begin{align*}
         \frac{-x(x-1)^2 }{2x(1-x)(1-y)(x-y)}\frac{dy}{dx} & = \frac{1}{2y(1-y)}\frac{dy}{dx}+\frac{1}{2y(1-y)}\cdot\frac{x}{y}\frac{dy}{dx}+o(x^{\sigma-1}) \\
         & = \frac{1}{2y(1-y)}\frac{dy}{dx}+\frac{1-\sigma}{2J}x^{\sigma-1}+o(x^{\sigma-1}),\\
        \frac{(\theta_1-\theta_2+1)y}{2x(1-x)(1-y)(x-y)} &=\frac{-\theta_1+\theta_2-1}{2x(1-y)}+\frac{-\theta_1+\theta_2-1}{2x(1-y)}\cdot\frac{x}{y}+o(x^{\sigma-1}) \\
    &=\frac{-\theta_1+\theta_2-1}{2x}+\frac{(-\theta_1+\theta_2-1)y}{2x(1-y)} +\frac{-\theta_1+\theta_2-1}{2J}x^{\sigma-1}+o(x^{\sigma-1}),    
    \end{align*}
 the first of which follows from $\frac{1}{1-\frac{x}{y}} = 1+\frac{x}{y}+o(x^\sigma)$. 
%And similarly  
%     \begin{align*}
%     \frac{(\theta_1-\theta_2+1)y}{2x(1-x)(1-y)(x-y)} &=\frac{-\theta_1+\theta_2-1}{2x(1-y)}+\frac{-\theta_1+\theta_2-1}{2x(1-y)}\cdot\frac{x}{y}+o(x^{\sigma-1}) \\
%     &=\frac{-\theta_1+\theta_2-1}{2x}+\frac{(-\theta_1+\theta_2-1)y}{2x(1-y)} +\frac{-\theta_1+\theta_2-1}{2J}x^{\sigma-1}+o(x^{\sigma-1}).
%     \end{align*}%In the subsequent estimation, we take the same process for terms as $\frac{y}{x-y}$, but omit the explanation. 
    Thus, 
    \begin{align*}
       \begin{split}
          k_2(x) &= k_2^0\cdot {\rm exp}\left(\int^{x} l_2(\xi)d\xi\right)
         = k_2^0\cdot x^{\frac{-\theta_1+\theta_2-\sigma}{2}}
        (1+f_0(x))(1+\frac{\theta_2-\sigma}{2\sigma J}x^\sigma) + o(x^{\frac{-\theta_1+\theta_2+\sigma}{2}}).
       \end{split}
    \end{align*}
\end{proof}

\subsubsection{The limit of $A(x)$ in \eqref{eq:delta Phi0}}

\begin{lemma}\label{lemma, limit of A(x)}
     The entries of $\delta_2(\Phi_0)$ obtained in \eqref{eq:delta Phi0} are
     \begin{align}\label{deltaPhi0}
       (\Phi_0)_{ii} = -\theta_i, \quad  i=1,2, \quad \text{and} \quad
            (\Phi_0)_{12} = \frac{k_1^0}{k_2^0}\cdot\frac{\theta_1-\theta_2-\sigma}{2}, \quad 
            (\Phi_0)_{21} = \frac{k_2^0}{k_1^0}\cdot\frac{-\theta_1+\theta_2-\sigma}{2}.
     \end{align}
\end{lemma}
\begin{proof}
 According to \eqref{eq:delta Phi0}, we have that $A(x)_{ii}=-\theta_i, i=1,2$, and 
    \begin{align}
        \begin{split}
            A(x)_{12}&=x^{-\theta_1+\theta_2}\cdot \Omega_{12},\\
A(x)_{21}&=x^{\theta_1-\theta_2}\cdot\Omega_{21}.
        \end{split}
    \end{align}  
     Then by lemma \ref{lemma k1,k2},  the leading terms of $\Omega_{12}$ and $\Omega_{21}$ are respectively given by
    \begin{align}
 %   \begin{split}
        \Omega_{12}&=\frac{k_1^0}{k_2^0}\cdot x^{\theta_1-\theta_2}(\frac{1-\sigma}{2}+\frac{\theta_1-\theta_2-1}{2})\cdot(1+o(1))=\frac{k_1^0}{k_2^0}\cdot\frac{\theta_1-\theta_2-\sigma}{2} x^{\theta_1-\theta_2}+o(x^{\theta_1-\theta_2})\label{Omega12} , \\  \Omega_{21}&=\frac{k_2^0}{k_1^0}\cdot x^{\theta_2-\theta_1}(\frac{1-\sigma}{2}-\frac{\theta_1-\theta_2+1}{2})\cdot(1+o(1))=\frac{k_2^0}{k_1^0}\cdot\frac{-\theta_1+\theta_2-\sigma}{2} x^{-\theta_1+\theta_2}+o(x^{-\theta_1+\theta_2})\label{Omega21} .
%        \end{split}
    \end{align}
    Here for $\Omega_{12}$, we used
    \begin{align*}
        \frac{k_1(x)}{k_2(x)} &= \frac{k_1^0}{k_2^0}\cdot x^{\theta_1-\theta_2}\cdot(1+o(1)), \\
        \frac{f(\theta_1,\theta_2,\theta_3,\theta_\infty;x,y)}{2(1-x)y} &=\frac{(x^2-x)\frac{dy}{dx}+(\theta_2-\theta_1+1)y}{2(x-1)y}+o(1) = \frac{1-\sigma}{2}+\frac{\theta_1-\theta_2-1}{2}+o(1).
    \end{align*}  
    The estimation of $\Omega_{21}$ is similar to the one for $\Omega_{12}$. Substituting \eqref{Omega12} and \eqref{Omega21} into $A(x)$, and letting $x\rightarrow 0$, we obtain \eqref{deltaPhi0}.
\end{proof}

\subsubsection{The limit of $B(x)$ in \eqref{eq:Phi0}}\label{sec: limit of B(x)}

Now we compute the limit of $B(x)$ as $x\rightarrow 0$, and finish the proof of Theorem \ref{mainthm}. According to Lemma \ref{lemma, limit of A(x)}, we have
\begin{align}
    \delta_2(\Phi_0) = 
    \left(
    \begin{matrix}
        k_1^0  & (\theta_1-\theta_2-\sigma)k_1^0 \\
        k_2^0  & (\theta_1-\theta_2+\sigma)k_2^0
    \end{matrix}
    \right)
    \left(
    \begin{matrix}
        -\frac{\theta_1+\theta_2+\sigma}{2} &0\\
        0   &-\frac{\theta_1+\theta_2-\sigma}{2}
    \end{matrix}
    \right)
    \left(
    \begin{matrix}
        k_1^0  & (\theta_1-\theta_2-\sigma)k_1^0 \\
        k_2^0  & (\theta_1-\theta_2+\sigma)k_2^0
    \end{matrix}
    \right)^{-1}.
\end{align}
Thus, 
\begin{equation*}
\begin{aligned}
    \left(x^{-\delta_2(\Phi_0)}\right)_{11} &= \ \frac{1}{2\sigma}(-\theta_1+\theta_2+\sigma)
        x^{\frac{\theta_1+\theta_2-\sigma}{2}}+
        \frac{1}{2\sigma}(\theta_1-\theta_2+\sigma) 
        x^{\frac{\theta_1+\theta_2+\sigma}{2}},\\
    \left(x^{-\delta_2(\Phi_0)}\right)_{12} &=  \  \frac{k_1^0}{2k_2^0\sigma}(\theta_1-\theta_2-\sigma)
        x^{\frac{\theta_1+\theta_2-\sigma}{2}}+
        \frac{k_1^0}{2k_2^0\sigma}(-\theta_1+\theta_2+\sigma) 
        x^{\frac{\theta_1+\theta_2+\sigma}{2}},\\
    \left(x^{-\delta_2(\Phi_0)}\right)_{21} &=  \  \frac{k_2^0}{2k_1^0\sigma}(-\theta_1+\theta_2-\sigma)
        x^{\frac{\theta_1+\theta_2-\sigma}{2}}+
        \frac{k_2^0}{2k_1^0\sigma}(\theta_1-\theta_2+\sigma) 
        x^{\frac{\theta_1+\theta_2+\sigma}{2}},\\
    \left(x^{-\delta_2(\Phi_0)}\right)_{22} &= \ \frac{1}{2\sigma}(\theta_1-\theta_2+\sigma)
        x^{\frac{\theta_1+\theta_2-\sigma}{2}}+
        \frac{1}{2\sigma}(-\theta_1+\theta_2+\sigma) 
        x^{\frac{\theta_1+\theta_2+\sigma}{2}}.
\end{aligned}
\end{equation*}
By a direct computation, we have as $x\rightarrow 0$ 
\begin{equation}\label{B1}
\begin{aligned}
%    \begin{split}
        (B(x))_{13}&=\left\{\frac{-\theta_1+\theta_2+\sigma}{2\sigma}\cdot\left(x^{\frac{-\theta_1+\theta_2-\sigma}{2}}\Omega_{13}-\frac{k_1^0}{k_2^0}\cdot x^{\frac{\theta_1-\theta_2-\sigma}{2}}\Omega_{23}\right)\right.\\
        &\left.\quad +\left(x^{\frac{-\theta_1+\theta_2+\sigma}{2}}\frac{\theta_1-\theta_2+\sigma}{2\sigma}\Omega_{13}+ x^{\frac{\theta_1-\theta_2+\sigma}{2}}\frac{k_1^0(-\theta_1+\theta_2+\sigma)}{k_2^0\cdot 2\sigma}\Omega_{23}\right) \right\}\cdot\Big(1+O(x)\Big) ,
        \end{aligned}
\end{equation}
\begin{equation}\label{B2}
\begin{aligned}
        (B(x))_{31}&=\left\{\frac{\theta_1-\theta_2+\sigma}{2\sigma}\cdot\left(x^{\frac{\theta_1-\theta_2-\sigma}{2}}\Omega_{31}+\frac{k_2^0}{k_1^0}\cdot x^{\frac{-\theta_1+\theta_2-\sigma}{2}}\Omega_{32}\right)\right.\\
        &\left.\quad +\left(x^{\frac{\theta_1-\theta_2+\sigma}{2}}\cdot\frac{-\theta_1+\theta_2+\sigma}{2\sigma}\Omega_{31}+ x^{\frac{-\theta_1+\theta_2+\sigma}{2}}\cdot\frac{k_2^0(-\theta_1+\theta_2-\sigma)}{k_1^0\cdot 2\sigma}\Omega_{32}\right) \right\}\cdot\Big(1+O(x)\Big) , 
\end{aligned}
\end{equation}

\begin{equation}\label{B3}
\begin{aligned}
        (B(x))_{32}&=\left\{\frac{-\theta_1+\theta_2+\sigma}{2\sigma}\cdot\left(x^{\frac{\theta_1-\theta_2-\sigma}{2}}\frac{k_1^0}{k_2^0}\cdot\Omega_{31}+ x^{\frac{-\theta_1+\theta_2-\sigma}{2}}\Omega_{32}\right)\right.\\
        &\left.\quad +\left(x^{\frac{\theta_1-\theta_2+\sigma}{2}}\cdot\frac{k_1^0(\theta_1-\theta_2-\sigma)}{k_2^0\cdot 2\sigma}\Omega_{31}+ x^{\frac{-\theta_1+\theta_2+\sigma}{2}}\cdot\frac{\theta_1-\theta_2+\sigma}{2\sigma}\Omega_{32}\right) \right\}\cdot\Big(1+O(x)\Big),
\end{aligned}
\end{equation}

%    \end{split}\\
%    \begin{split}
\begin{equation}\label{B4}
\begin{aligned}
        (B(x))_{23}&=\left\{\frac{-\theta_1+\theta_2-\sigma}{2\sigma}\cdot\left(x^{\frac{-\theta_1+\theta_2-\sigma}{2}}\frac{k_2^0}{k_1^0}\cdot\Omega_{13}- x^{\frac{\theta_1-\theta_2-\sigma}{2}}\Omega_{23}\right)\right.\\
        &\left.\quad +\left(x^{\frac{-\theta_1+\theta_2+\sigma}{2}}\cdot\frac{k_2^0(\theta_1-\theta_2+\sigma)}{k_1^0\cdot 2\sigma}\Omega_{13}+ x^{\frac{\theta_1-\theta_2+\sigma}{2}}\cdot\frac{-\theta_1+\theta_2+\sigma}{2\sigma}\Omega_{23}\right) \right\}\cdot\Big(1+O(x)\Big). 
%    \end{split}
\end{aligned}
\end{equation}

Similar to the proof of Lemma \ref{lemma k1,k2}, we have the following estimates for $\Omega(x)$

\begin{lemma}\label{lemma Omega}
The entries $\Omega(x)_{13},\ \Omega(x)_{23},\ \Omega(x)_{31},\ \Omega(x)_{32}$ have the following asymptotic expansions as $x\rightarrow 0$
    \begin{equation*}
    \begin{aligned}
            \Omega_{13}&=k_1^0 x^{\frac{\theta_1-\theta_2-\sigma}{2}}(1+f_0(x))\left(1-\frac{\theta_1}{2\sigma J}x^\sigma\right)\left(f_1(x)+\frac{-\theta_3-\theta_\infty+\sigma}{2}+\frac{\theta_1}{2J}x^{\sigma}\right)+o(x^{\frac{\theta_1-\theta_2+\sigma}{2}}),  \\
%    \begin{split}
         \Omega_{23}&=k_2^0 x^{\frac{\theta_2-\theta_1-\sigma}{2}}(1+f_0(x))\left(1+\frac{\theta_2-\sigma}{2\sigma J}x^{\sigma}\right)\left(f_1(x)+\frac{-\theta_3-\theta_\infty+\sigma}{2}+\frac{-\theta_2+\sigma}{2J}x^{\sigma}\right)+o(x^{\frac{\theta_1-\theta_2+\sigma}{2}}) ,\\
%    \end{split} \\
            \Omega_{31}&=\frac{1}{k_1^0}x^{\frac{\theta_2-\theta_1-\sigma}{2}}(1+f_0(x))^{-1}\left(1+\frac{\theta_1}{2\sigma J}x^\sigma\right)\left(f_2(x)+\frac{(\theta_3-\theta_\infty+\sigma)J}{2}-\frac{\theta_1}{2}x^{\sigma}\right)+o(x^{\frac{\theta_1-\theta_2+\sigma}{2}}) ,  \\
%            \begin{split}
            \Omega_{32}&=\frac{1}{k_2^0}x^{\frac{\theta_1-\theta_2-\sigma}{2}}(1+f_0(x))^{-1}\left(1-\frac{\theta_2-\sigma}{2\sigma J}x^\sigma\right)\left(-f_2(x)-\frac{(\theta_3-\theta_\infty+\sigma)J}{2}-\frac{\theta_2-\theta_3+\theta_\infty}{2}x^{\sigma}\right)+o(x^{\frac{\theta_1-\theta_2+\sigma}{2}}) . 
%            \end{split}
   \end{aligned}
\end{equation*}
Here 
   \[f_1(x)=-\frac{x}{2y}\frac{dy}{dx}-\frac{\sigma-1}{2}-\frac{(1-\theta_{\infty})y}{2}=o(1), \quad \text{as } x\rightarrow 0,\] 
   and 
   \[f_2(x)=\left(-\frac{1}{2(1-y)}\frac{dy}{dx}-\frac{(1-\theta_{\infty})y^2}{2x(1-y)}\right)x^\sigma+\frac{(1-\sigma)J}{2}+\frac{\theta_3-\theta_\infty+1}{2}\left(\frac{y}{(1-y)x}x^\sigma-J\right)=o(1), \quad \text{as } x\rightarrow 0.\] 
\end{lemma}

\begin{proof}
    We give the estimations for $\Omega_{13}$ and $\Omega_{23}$, and the other two terms are estimated similarly. From estimation for each term of $f(\theta_1, \theta_1-\theta_3 - \theta_{\infty}, \theta_3, \theta_{\infty}; x, y)$ and $f(-\theta_2 + \theta_3 + \theta_{\infty}, -\theta_2, \theta_3, \theta_{\infty}; x, y)$, similar to the proof of Lemma \ref{lemma k1,k2}, we have 
    \begin{align}
            \Omega_{13}&= k_1(x)\cdot\left(\left(-\frac{x}{2y}\frac{dy}{dx}-\frac{(1-\theta_{\infty})y}{2}+o(x^\sigma)\right)+\left(\frac{-\theta_3-\theta_{\infty}+1}{2}+O(x)\right)+\left(\frac{\theta_1}{2J}x^\sigma+o(x^\sigma)\right)\right)\nonumber\\
            &= k_1(x)\cdot\left(\left(-\frac{x}{2y}\frac{dy}{dx}-\frac{\sigma-1}{2}-\frac{(1-\theta_{\infty})y}{2}\right)+\frac{-\theta_3-\theta_{\infty}+\sigma}{2}+\frac{\theta_1}{2J}x^\sigma+o(x^\sigma)\right),\\
        \Omega_{23}&= k_2(x)\cdot\left(-\frac{x}{2y}\frac{dy}{dx}(1+\frac{x}{y})-\frac{(1-\theta_{\infty})y}{2}+o(x^\sigma)+\frac{-\theta_3-\theta_{\infty}+1}{2}(1+\frac{x}{y})+O(x)+
        \frac{\theta_{\infty}-\theta_2+\theta_3}{2J}x^\sigma+o(x^\sigma)\right)\nonumber\\ 
            &= k_2(x)\cdot\left(\left(-\frac{x}{2y}\frac{dy}{dx}-\frac{\sigma-1}{2}-\frac{(1-\theta_{\infty})y}{2}\right)+\frac{-\theta_3-\theta_\infty+\sigma}{2}+\frac{-\theta_2+\sigma}{2J}x^{\sigma}+o(x^\sigma)\right).
    \end{align}
Here for $\Omega_{13}$, we used 
    \begin{align*}
        \frac{(1-\theta_\infty)y^2}{2(x-1)y} &= -\frac{(1-\theta_\infty)y}{2}+o(x^\sigma), \\
        \frac{((\theta_1-\theta_3)x+\theta_\infty+\theta_3-1)y}{2(x-1)y}&=\frac{-\theta_3-\theta_\infty+1}{2}+O(x), \\
        -\frac{\theta_1x}{2(x-1)y} &= \frac{\theta_1}{2J}x^\sigma+o(x^\sigma).
    \end{align*}
And for $\Omega_{23}$, we used
\begin{align*}
    \frac{(x-x^2)}{2(x-y)}\frac{dy}{dx} &= -\frac{x}{2y}\frac{dy}{dx}(1+\frac{x}{y})+o(x^\sigma) = -\frac{x}{2y}\frac{dy}{dx}+\frac{\sigma-1}{2J}x^\sigma+o(x^\sigma),\\
    \frac{((\theta_\infty-\theta_2)x+\theta_\infty+\theta_3-1)y}{2(x-y)}&= \frac{-\theta_3-\theta_\infty+1}{2}(1+\frac{x}{y})+O(x) = \frac{-\theta_3-\theta_\infty+1}{2}+\frac{-\theta_3-\theta_\infty+1}{2J}x^\sigma+o(x^\sigma).
\end{align*}
    Assume $f_1(x)=-\frac{x}{2y}\frac{dy}{dx}-\frac{\sigma-1}{2}-\frac{(1-\theta_{\infty})y}{2} $, then $f_1(x)= o(1) \text{ as}\ x\rightarrow 0$. Using the estimation in Lemma \ref{lemma k1,k2} for $k_1(x)$ and $k_2(x)$, we obtain the desired expansion of $\Omega_{13}$ and $\Omega_{23}$.   
\end{proof}

Now we are ready to finish the proof of Theorem \ref{mainthm}.

\begin{proof}[Proof of Theorem \ref{mainthm}]
We compute the limit of $(B(x))_{13}$ as $x\rightarrow 0$, and the other three terms are similar. Following \eqref{B1}, 
\begin{equation*}%\label{B1}
\begin{aligned}
%    \begin{split}
        (B(x))_{13}&=\left\{\frac{-\theta_1+\theta_2+\sigma}{2\sigma}\cdot\left(x^{\frac{-\theta_1+\theta_2-\sigma}{2}}\Omega_{13}-\frac{k_1^0}{k_2^0}\cdot x^{\frac{\theta_1-\theta_2-\sigma}{2}}\Omega_{23}\right)\right.\\
        &\left.\quad +\left(x^{\frac{-\theta_1+\theta_2+\sigma}{2}}\frac{\theta_1-\theta_2+\sigma}{2\sigma}\Omega_{13}+ x^{\frac{\theta_1-\theta_2+\sigma}{2}}\frac{k_1^0(-\theta_1+\theta_2+\sigma)}{k_2^0\cdot 2\sigma}\Omega_{23}\right) \right\}\cdot\Big(1+O(x)\Big) .
        \end{aligned}
\end{equation*}
Using Lemma \ref{lemma Omega}, we have
       \begin{align*}
        & \quad x^{\frac{-\theta_1+\theta_2-\sigma}{2}}\Omega_{13}-\frac{k_1^0}{k_2^0}\cdot x^{\frac{\theta_1-\theta_2-\sigma}{2}}\Omega_{23}\\
        & =
        k_1^0\cdot x^{-\sigma}(1+f_0(x))\left(\frac{\theta_1+\theta_2-\theta_3-\theta_\infty}{2J}\right)x^\sigma +k_1^0(1+f_0(x))\left(\ \frac{(\theta_1+\theta_2-\sigma)(\theta_3+\theta_\infty-\sigma)}{4\sigma J}\right)\\
        &\quad +k_1^0(1+f_0(x))\frac{\theta_3+\theta_{\infty}-\sigma}{2J} + o(1), \\
        &\rightarrow k_1^0\frac{(\theta_1+\theta_2-\sigma)
        (\theta_3+\theta_{\infty}+\sigma)}{4\sigma J },\end{align*}
and
          \begin{align*}
         &\quad x^{\frac{-\theta_1+\theta_2+\sigma}{2}}\cdot\frac{\theta_1-\theta_2+\sigma}{2\sigma}\Omega_{13}+x^{\frac{\theta_1-\theta_2+\sigma}{2}}\frac{k_1^0(-\theta_1+\theta_2+\sigma)}{k_2^0\cdot 2\sigma}\Omega_{23} \\
        & \rightarrow {k_1^0}\cdot \frac{\theta_1-\theta_2+\sigma}{2\sigma}\cdot\frac{-\theta_3-\theta_\infty+\sigma}{2} + {k_1^0}\cdot \frac{-\theta_1+\theta_2+\sigma}{2\sigma}\cdot\frac{-\theta_3-\theta_\infty+\sigma}{2} \\
         &= {k_1^0}\frac{-\theta_3-\theta_\infty+\sigma}{2}.
    \end{align*}

    Thus as $x\rightarrow 0$, \begin{align*}
        (B(x))_{13} &\rightarrow \frac{-\theta_1+\theta_2+\sigma}{2\sigma}\cdot k_1^0\frac{(\theta_1+\theta_2-\sigma)
        (\theta_3+\theta_{\infty}+\sigma)}{4\sigma J } + {k_1^0}\frac{-\theta_3-\theta_\infty+\sigma}{2} \\
        & = \frac{k_1^0}{2}\cdot(-\theta_3-\theta_\infty+\sigma)-\frac{k_1^0}{8\sigma^2 J}\cdot(\theta_1-\theta_2-\sigma)(\theta_1+\theta_2-\sigma)(\theta_3+\theta_{\infty}+\sigma).
    \end{align*}

\end{proof}

\subsection{A new interpretation of Jimbo's formula}\label{commdiag}
In this subsection, let us assume $y(x;\sigma,J, \theta_1,\theta_2,\theta_3,\theta_\infty)$ is a generic solution of the Painlev\'e VI equation \eqref{eq:PVI}, and $\Phi(u;\Phi_0)$ a corresponding solution of \eqref{eq:isomonodromy3}, satisfying the conditions \eqref{eq:restr of diag of phi} and \eqref{eq:restr of eigen of phi}. Let us denote by $p_{ij}$ (defined in \eqref{eq:def of pij} ) and $\theta_k \ (k=1,2,3,\infty)$ the monodromy parameters of the corresponding Fuchsion system \eqref{eq:2by2system}, and denote by $S_\pm(u,\Phi(u;\Phi_0))$ the Stokes matrices of the system \eqref{eq:3by3system}. Following the above discussion, we have the following diagram with arrows/correspondences $F, Q, G, P$

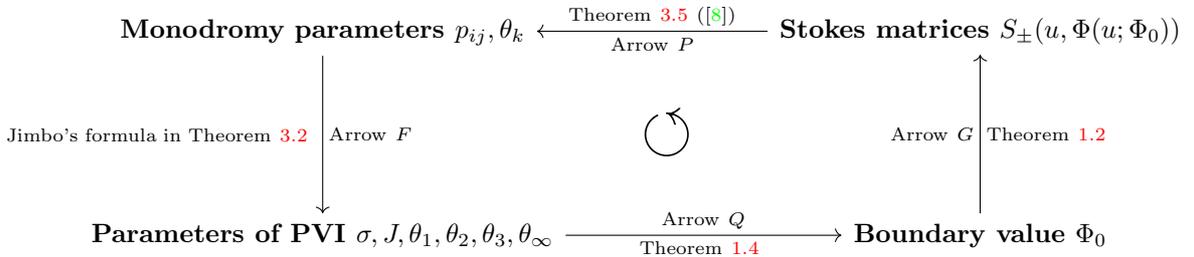
\[
\begin{tikzcd}\label{diagram: commu diagram}
{\textbf{Monodromy parameters } p_{ij},\theta_k} \ar[from= 1-1,to=3-1, "\text{Jimbo's formula in Theorem}\ \ref{thm Jimbo's formua for leading term of PVI}\ "',"\text{Arrow} \ F"] &  {}%\ar[from=1-3,to=1-1, "\cite{Degano-Guzzetti2023}\text{Theorem}\ \ref{thm: relation of pij and S+-}\ "'," \text{Arrow}\ P\ "]  
\ar[from=1-3, to=1-1,"\text{Theorem}\ \ref{thm: relation of pij and S+-}\ (\text{\cite{Degano-Guzzetti2023}})"',"\text{Arrow}\ P"]
&  \textbf{Stokes matrices  }  S_\pm(u,\Phi(u;\Phi_0)) \ar[from=3-3,to=1-3, "\text{Theorem}\ \ref{thm: introcatformula}\ "',"\text{Arrow}\ G"] \\
                                 &  {} \text{\Huge$\circlearrowleft$}   & \\
{\textbf{Parameters of PVI }  \sigma, J, \theta_1,\theta_2,\theta_3,\theta_\infty }  &  {} \ar[from=3-1,to=3-3, "\text{Theorem} \ \ref{mainthm}\ "',"\text{Arrow}\ Q"] &  \textbf{Boundary value } \Phi_0                    
\end{tikzcd}
\]
Here following Theorem \ref{thm: introcatformula} of \cite{xu2019closure1}, the entries of $3\times 3$ Stokes matrices $S_{\pm}$ (Arrow $G$) are
 \begin{align*}
     (S_{+})_{kk} &= {\rm e}^{{\rm i}\pi\theta_k},\ k=1,2,3,\\
     (S_{+})_{12} &= -2{\rm i}\pi\cdot e^{-{\rm i}\pi\lambda_{1}^{(1)}}\cdot\frac{(\Phi_0)_{12}}{\Gamma(1+\lambda_1^{(1)}-\lambda_1^{(2)})\cdot\Gamma(1+\lambda_1^{(1)}-\lambda_2^{(2)})},\\
    (S_+)_{23} &=  2{\rm i}\pi\cdot e^{-{\rm i}\pi{\lambda_2^{(1)}}}\cdot \frac{\Gamma(1+{\lambda_1^{(2)}-\lambda_2^{(2)}})\Gamma({\lambda_1^{(2)}-\lambda_2^{(2)}})}
    {\prod_{j=1}^3\Gamma(1+{\lambda_1^{(2)}-\lambda_j^{(3)}})}\cdot \frac{\Delta_{1,3}^{1,2}\left({\Phi_0-\lambda_1^{(2)}}\right)}{\Gamma(1+\lambda_1^{(2)}-\lambda_1^{(1)})} \\
    &\quad + 2{\rm i}\pi\cdot e^{-{\rm i}\pi{\lambda_2^{(1)}}}\cdot \frac{\Gamma(1+{\lambda_2^{(2)}-\lambda_1^{(2)}})\Gamma({\lambda_2^{(2)}-\lambda_1^{(2)}})}
    {\prod_{j=1}^3\Gamma(1+{\lambda_2^{(2)}-\lambda_j^{(3)}})}\cdot \frac{\Delta_{1,3}^{1,2}\left({\Phi_0-\lambda_2^{(2)}}\right)}{\Gamma(1+\lambda_2^{(2)}-\lambda_1^{(1)})},  \\
      (S_+)_{13} &= - 2{\rm i}\pi\cdot e^{-{\rm i}\pi{\lambda_1^{(2)}}}\cdot \frac{\Gamma(1+{\lambda_1^{(2)}-\lambda_2^{(2)}})\Gamma({\lambda_1^{(2)}-\lambda_2^{(2)}})}
    {\prod_{j=1}^3\Gamma(1+{\lambda_1^{(2)}-\lambda_j^{(3)}})\cdot \Gamma(1+{\lambda_1^{(1)}-\lambda_2^{(2)}})}\cdot \frac{(\Phi_0)_{12}\cdot \Delta_{1,3}^{1,2}\left({\Phi_0-\lambda_1^{(2)}}\right)}{(\lambda_1^{(1)}-\lambda_1^{(2)})}\\
    &\quad -2{\rm i}\pi\cdot e^{-{\rm i}\pi{\lambda_2^{(2)}}}\cdot \frac{\Gamma(1+{\lambda_2^{(2)}-\lambda_1^{(2)}})\Gamma({\lambda_2^{(2)}-\lambda_1^{(2)}})}
    {\prod_{j=1}^3\Gamma(1+{\lambda_2^{(2)}-\lambda_j^{(3)}})\cdot \Gamma(1+{\lambda_1^{(1)}-\lambda_1^{(2)}})}\cdot \frac{(\Phi_0)_{12}\cdot \Delta_{1,3}^{1,2}\left({\Phi_0-\lambda_2^{(2)}}\right)}{(\lambda_1^{(1)}-\lambda_2^{(2)})},
\end{align*}  
and
\begin{align*}
(S_{-})_{kk} &= {\rm e}^{{\rm i}\pi\theta_k},\ k=1,2,3,\\
     (S_{-})_{21} &= -2{\rm i}\pi\cdot e^{-{\rm i}\pi\lambda_{2}^{(1)}}\cdot\frac{(\Phi_0)_{21}}{\Gamma(1+\lambda_1^{(2)}-\lambda_1^{(1)})\cdot\Gamma(1+\lambda_2^{(2)}-\lambda_1^{(1)})},\\
    (S_-)_{32} &=  -2{\rm i}\pi\cdot e^{-{\rm i}\pi{\lambda_3^{(2)}}}\cdot \frac{\Gamma(1+{\lambda_2^{(2)}-\lambda_1^{(2)}})\Gamma({\lambda_2^{(2)}-\lambda_1^{(2)}})}
    {\prod_{j=1}^3\Gamma(1+\lambda_j^{(3)}-\lambda_1^{(2)})}\cdot \frac{\Delta_{1,2}^{1,3}\left({\Phi_0-\lambda_1^{(2)}}\right)}{\Gamma(1+\lambda_1^{(1)}-\lambda_1^{(2)})} \\ 
    &\quad -2{\rm i}\pi\cdot e^{-{\rm i}\pi{\lambda_3^{(2)}}}\cdot \frac{\Gamma(1+{\lambda_1^{(2)}-\lambda_2^{(2)}})\Gamma({\lambda_1^{(2)}-\lambda_2^{(2)}})}
    {\prod_{j=1}^3\Gamma(1+\lambda_j^{(3)}-\lambda_2^{(2)})}\cdot \frac{\Delta_{1,2}^{1,3}\left({\Phi_0-\lambda_2^{(2)}}\right)}{\Gamma(1+\lambda_1^{(1)}-\lambda_2^{(2)})}, \\
    (S_-)_{31} &=  2{\rm i}\pi\cdot e^{{\rm i}\pi({\lambda_1^{(1)}-\lambda_1^{(2)}-\lambda_3^{(2)}})}\cdot \frac{\Gamma(1+{\lambda_2^{(2)}-\lambda_1^{(2)}})\Gamma({\lambda_2^{(2)}-\lambda_1^{(2)}})}
    {\prod_{j=1}^3\Gamma(1+{\lambda_j^{(3)}-\lambda_1^{(2)}})\cdot \Gamma(1+{\lambda_2^{(2)}-\lambda_1^{(1)}})}\cdot \frac{(\Phi_0)_{21}\cdot \Delta_{1,2}^{1,3}\left({\Phi_0-\lambda_1^{(2)}}\right)}{(\lambda_1^{(1)}-\lambda_1^{(2)})} \\ \nonumber
    &\quad 2{\rm i}\pi\cdot e^{{\rm i}\pi({\lambda_1^{(1)}-\lambda_2^{(2)}-\lambda_3^{(2)}})}\cdot \frac{\Gamma(1+{\lambda_1^{(2)}-\lambda_2^{(2)}})\Gamma({\lambda_1^{(2)}-\lambda_2^{(2)}})}
    {\prod_{j=1}^3\Gamma(1+{\lambda_j^{(3)}-\lambda_2^{(2)}})\cdot \Gamma(1+{\lambda_1^{(2)}-\lambda_1^{(1)}})}\cdot \frac{(\Phi_0)_{21}\cdot \Delta_{1,2}^{1,3}\left({\Phi_0-\lambda_2^{(2)}}\right)}{(\lambda_1^{(1)}-\lambda_2^{(2)})},
    \end{align*}
    where $\{\lambda_i^{(k)}\}_{i=1,...,k,k+1}$ are defined in Theorem \ref{thm: introcatformula}.

Beginning from any vertex in the diagram and following the arrows in the directions of $F, Q, G, P$ to return the same vertex, we obtain identity maps. In particular the composition 
\[ P \circ G \circ Q(\theta_1,\theta_2,\theta_3,\theta_4,\sigma,J)\rightarrow (p_{ij},\theta_k)\]
gives (the inverse of) Jimbo's formula in Theorem \ref{thm Jimbo's formua for leading term of PVI}, which we now check by hand as an example.

First, by a direction computation we get

\begin{lemma}
\label{prop:express pij in Stokes}
Under the composition $G\circ Q$ of the correspondences, there exist unique numbers $k_1^0,\ k_2^0\in \mathbb{C}\setminus\{0\}$ such that the Stokes matrices are expressed via the  parameters of PVI $(\sigma,J,\theta_1,\theta_2,\theta_3,\theta_\infty)$ by
     \begin{align}\label{eq: S+12}
      (S_{+})_{kk}& = {\rm e}^{{\rm i}\pi\theta_k},\qquad
       (S^{-1}_{-})_{kk} = {\rm e}^{-{\rm i}\pi\theta_k},\ k=1,2,3,\\
    \left(S_+\right)_{12}&=-\frac{k_1^0}{k_2^0} \cdot 2\mathrm{i}\pi \ \mathrm{e}^{\mathrm{i}\pi{\theta_1}}\cdot \frac{\theta_1-\theta_2-\sigma}{2\Gamma\left(1+\frac{\theta_2-\theta_1+\sigma}{2}\right)\Gamma\left(1+\frac{\theta_2-\theta_1-\sigma}{2}\right)},\\ \label{eq:S_12}
     \left(S_-^{-1}\right)_{21}&=-\frac{k_2^0}{k_1^0}\cdot 2\mathrm{i}\pi \ \mathrm{e}^{-\mathrm{i}\pi{\theta_1}}\cdot \frac{\theta_1-\theta_2+\sigma}{2\Gamma\left(1-\frac{\theta_2-\theta_1+\sigma}{2}\right)\Gamma\left(1-\frac{\theta_2-\theta_1-\sigma}{2}\right)},
\end{align}
and
\begin{align}
    (S_+)_{23}&={k_2^0}\cdot 2\mathrm{i}\pi \  \mathrm{e}^{\mathrm{i}\pi{\theta_2}}\cdot((S_+)_{23}^{(1)}+(S_+)_{23}^{(2)}), \label{eq:S+23}\\
% \end{align}
% \begin{align}
    (S_-^{-1})_{32}&=\frac{1}{k_2^0}\cdot 2\mathrm{i}\pi\ \mathrm{e}^{-\mathrm{i}\pi{\theta_2}}\cdot((S_-)_{32}^{(1)}+(S_-)_{32}^{(2)}). \label{eq:S-32}
\end{align}
Here
\begin{equation*}
    \begin{aligned}
           (S_+)_{23}^{(1)} &= \frac{(\Gamma(1-\sigma))^2(\theta_3+\theta_{\infty}-\sigma)}{2\Gamma(1-\frac{\theta_1+\theta_2+\sigma}{2})\Gamma(1+\frac{\theta_1-\theta_2-\sigma}{2})\Gamma(1+\frac{\theta_3+\theta_{\infty}-\sigma}{2})\Gamma(1+\frac{\theta_3-\theta_{\infty}-\sigma}{2})},\\
           (S_-)_{32}^{(1)} &= \frac{(\Gamma(\sigma))^2\cdot(\theta_2-\theta_1+\sigma)(\theta_1+\theta_2+\sigma)(\theta_3-\theta_{\infty}-\sigma)}{8\Gamma(1+\frac{\theta_1+\theta_2+\sigma}{2})\Gamma(1-\frac{\theta_1-\theta_2-\sigma}{2})\Gamma(1-\frac{\theta_3+\theta_{\infty}-\sigma}{2})\Gamma(1-\frac{\theta_3-\theta_{\infty}-\sigma}{2})},\\
                 (S_+)_{23}^{(2)} &= \frac{(\Gamma(\sigma))^2\cdot(\theta_1-\theta_2+\sigma)(\theta_1+\theta_2-\sigma)(\theta_3+\theta_{\infty}+\sigma)}{8 J\cdot\Gamma(1-\frac{\theta_1+\theta_2-\sigma}{2})\Gamma(1+\frac{\theta_1-\theta_2+\sigma}{2})\Gamma(1+\frac{\theta_3+\theta_{\infty}+\sigma}{2})\Gamma(1+\frac{\theta_3-\theta_{\infty}+\sigma}{2})}, \\
      % &= \frac{(\Gamma(1+\sigma))^2\cdot(\theta_1-\theta_2+\sigma)(\theta_1+\theta_2-\sigma)(\theta_3+\theta_{\infty}-\sigma)\cdot c\cdot d}
      % {2(\theta_1+\theta_2+\sigma)(\theta_2-\theta_1+\sigma)(a+b)\Gamma(1-\frac{\theta_1+\theta_2-\sigma}{2})\Gamma(1+\frac{\theta_1-\theta_2+\sigma}{2})\Gamma(1+\frac{\theta_3+\theta_{\infty}+\sigma}{2})\Gamma(1+\frac{\theta_3-\theta_{\infty}+\sigma}{2})},\\
      (S_-)_{32}^{(2)} &= \frac{J\cdot (\Gamma(1-\sigma))^2(-\theta_3+\theta_{\infty}-\sigma)}{2\Gamma(1+\frac{\theta_1+\theta_2-\sigma}{2})\Gamma(1-\frac{\theta_1-\theta_2+\sigma}{2})\Gamma(1-\frac{\theta_3+\theta_{\infty}+\sigma}{2})\Gamma(1-\frac{\theta_3-\theta_{\infty}+\sigma}{2})} .
        % &= \frac{ (\Gamma(-\sigma))^2(\theta_1+\theta_2+\sigma)(\theta_1-\theta_2-\sigma)(\theta_3+\theta_\infty+\sigma)(\theta_3-\theta_{\infty}+\sigma)\cdot d}
        % {8c(a+b)(\theta_3+\theta_\infty-\sigma)\Gamma(1+\frac{\theta_1+\theta_2-\sigma}{2})\Gamma(1-\frac{\theta_1-\theta_2+\sigma}{2})\Gamma(1-\frac{\theta_3+\theta_{\infty}+\sigma}{2})\Gamma(1-\frac{\theta_3-\theta_{\infty}+\sigma}{2})}.
    \end{aligned}
\end{equation*}

 The entries $(S_{+})_{13}$ and $(S_{-}^{-1})_{31}$ have similar explicit expressions.
     \end{lemma}

\begin{prop}
    The explicit expressions of composition $P \circ G \circ Q(\sigma,J,\theta_1,\theta_2,\theta_3,\theta_\infty)\rightarrow (p_{ij},\theta_k)$ are as follows
    \begin{align*}
    p_{12} &=2\cos{(\pi\sigma)},\\
    p_{23} &= \frac{2}{\sin^2{(\pi\sigma)}}(\cos{(\pi\theta_1)}\cos{(\pi\theta_\infty)}+\cos{(\pi\theta_2)}\cos{(\pi\theta_3)}-\cos{(\pi\theta_1)}\cos{(\pi\theta_3)}\cos{(\pi\sigma)}-\cos{(\pi\theta_2)}\cos{(\pi\theta_\infty)}\cos{(\pi\sigma)})\\
    & \quad + 4\frac{\sin{\frac{\pi(\theta_1+\theta_2+\sigma)}{2}}\sin{\frac{\pi(\theta_1-\theta_2-\sigma)}{2}}\sin{\frac{\pi(\theta_3+\theta_{\infty}+\sigma)}{2}}\sin{\frac{\pi(\theta_3-\theta_{\infty}+\sigma)}{2}}}{\sin^2{(\pi\sigma)}}L\cdot J \\
    &\quad + 4\frac{\sin{\frac{\pi(\theta_1+\theta_2-\sigma)}{2}}\sin{\frac{\pi(\theta_1-\theta_2+\sigma)}{2}}\sin{\frac{\pi(\theta_3+\theta_{\infty}-\sigma)}{2}}\sin{\frac{\pi(\theta_3-\theta_{\infty}-\sigma)}{2}}}{\sin^2{(\pi\sigma)}}\frac{1}{L}\cdot \frac{1}{J},\\
     p_{13} &=\frac{2}{\sin^2{(\pi\sigma)}}(\cos{(\pi\theta_1)}\cos{(\pi\theta_3)}+\cos{(\pi\theta_2)}\cos{(\pi\theta_\infty)}-\cos{(\pi\theta_2)}\cos{(\pi\theta_3)}\cos{(\pi\sigma)}-\cos{(\pi\theta_1)}\cos{(\pi\theta_\infty)}\cos{(\pi\sigma)})\\
    &\quad -4{\rm e^{{\rm i}\pi\sigma}}\frac{\sin{\frac{\pi(\theta_1+\theta_2+\sigma)}{2}}\sin{\frac{\pi(\theta_1-\theta_2-\sigma)}{2}}\sin{\frac{\pi(\theta_3+\theta_{\infty}+\sigma)}{2}}\sin{\frac{\pi(\theta_3-\theta_{\infty}+\sigma)}{2}}}{\sin^2{(\pi\sigma)}}L\cdot J\\
    &\quad-4{\rm e}^{{-\rm i}\pi\sigma}\frac{\sin{\frac{\pi(\theta_1+\theta_2-\sigma)}{2}}\sin{\frac{\pi(\theta_1-\theta_2+\sigma)}{2}}\sin{\frac{\pi(\theta_3+\theta_{\infty}-\sigma)}{2}}\sin{\frac{\pi(\theta_3-\theta_{\infty}-\sigma)}{2}}}{\sin^2{(\pi\sigma)}}\frac{1}{L}\cdot \frac{1}{J}.
    \end{align*}
Here \begin{align*}
    L &= 4\sigma^2\cdot\frac{\theta_\infty+\theta_3-\sigma}{(\theta_1+\theta_2+\sigma)(-\theta_1+\theta_2+\sigma)(\theta_\infty+\theta_3+\sigma)}\cdot c, 
\end{align*}
and $c$ is defined in \eqref{eq:a,b,c,d in J}. 
\end{prop}
\begin{proof}
   Substitute the expressions in Lemma \ref{prop:express pij in Stokes} into right sides of \eqref{eq:relation of pij and S+-}  in Theorem \ref{thm: relation of pij and S+-}, we can get the desired expressions of $P\circ G \circ Q$. \end{proof}

To prove that $P\circ G \circ Q$ gives the inverse of Jimbo's formula, we verify 
\begin{prop}
The composition \[(F \circ P \circ G) \circ Q = {\rm Id}: (\sigma,J,\theta_1,\theta_2,\theta_3,\theta_\infty)\rightarrow (\sigma,J,\theta_1,\theta_2,\theta_3,\theta_\infty).\]
\end{prop}
\begin{proof}     
  For $a,b,d$ defined in \eqref{eq:a,b,c,d in J}, we have
   \begin{align*}
       a+b &= \frac{{\rm i}\sin{\pi\sigma}}{2}\cdot({\rm e}^{{\rm i}\pi\sigma}p_{23}+p_{13})\\
       &\quad -{\rm e}^{{\rm i}\pi\sigma}(\cos{\pi\theta_2}\cos{\pi\theta_\infty}+\cos{\pi\theta_1}\cos{\pi\theta_3})+\cos{\pi\theta_2}\cos{\pi\theta_3}+\cos{\pi\theta_1}\cos{\pi\theta_\infty}\\
       &= \frac{{\rm i}}{2\sin{\pi\sigma}}\frac{d}{L}
       \cdot\frac{1}{J}\cdot(-{\rm e}^{{\rm i}\pi\sigma}+{\rm e}^{-{\rm i}\pi\sigma}) \\
       &\quad +\frac{{\rm i}}{\sin{\pi\sigma}}{\rm e}^{{\rm i}\pi\sigma}(\cos{\pi\theta_1}\cos{\pi\theta_\infty}+\cos{\pi\theta_2}\cos{\pi\theta_3}-\cos{\pi\theta_1}\cos{\pi\theta_3}\cos{\pi\sigma}-\cos{\pi\theta_2}\cos{\pi\theta_\infty}\cos{\pi\sigma})\\
       &\quad + \frac{{\rm i}}{\sin{\pi\sigma}}(\cos{\pi\theta_1}\cos{\pi\theta_3}+\cos{\pi\theta_2}\cos{\pi\theta_\infty}-\cos{\pi\theta_2}\cos{\pi\theta_3}\cos{\pi\sigma}-\cos{\pi\theta_1}\cos{\pi\theta_\infty}\cos{\pi\sigma})\\
       &\quad -{\rm e}^{{\rm i}\pi\sigma}(\cos{\pi\theta_2}\cos{\pi\theta_\infty}+\cos{\pi\theta_1}\cos{\pi\theta_3})+\cos{\pi\theta_2}\cos{\pi\theta_3}+\cos{\pi\theta_1}\cos{\pi\theta_\infty}\\
       &= \frac{d}{L}\cdot\frac{1}{J}.
   \end{align*}    
Thus in \eqref{eq:expansion of J}, 
   \begin{align*}
       \frac{(\theta_1+\theta_2+\sigma)(-\theta_1+\theta_2+\sigma)(\theta_\infty+\theta_3+\sigma)\cdot d}{4\sigma^2(\theta_\infty+\theta_3-\sigma)\cdot c\cdot(a+b)} = J.
   \end{align*}
   Other parameters can be obtained from the definitions directly.
\end{proof}

%%%%%%%%%%%%%%%%%%%%%%%%%%%%%%%%%%%%%%%%%%%%%%%%%%%%%%%%%%%%%%%%%%%%%%%%%%%%%%%%%%%%%%%%%%%%%%%%%%%%%%%%%%%%%%%%%%%%%%%%%%%%%%%%%%%%%%

%%%%%%%%%%%%%%%%%%%%%%%%%%%%%%%%%%%%%%%%%%%%%%%%%%%%%%%%%%%%%%%%%%%%%%%%%%%%%%%%%%%%%%%%%%%%%%%%%%%%%%%%%%%%%%%%%%%%%%%%%%%%%%%%%%%%%%

\Addresses
\end{document}